\documentclass[12pt]{amsart}
\usepackage[obeyspaces,hyphens,spaces]{url}
\usepackage{amsmath, amssymb, amsthm}
\usepackage{comment, empheq}
\usepackage{mathdots, breqn, xcolor}
\usepackage[colorlinks]{hyperref}
\usepackage{mathtools}
\usepackage[capitalize,nameinlink,noabbrev]{cleveref}
\usepackage{dsfont}
\numberwithin{equation}{section}
\newtheorem{theorem}{Theorem}
\newtheorem{lemma}{Lemma}[section]
\newtheorem{prop}[lemma]{Proposition}

\newtheorem{cor}[lemma]{Corollary}

\theoremstyle{definition}
\newtheorem{defn}[lemma]{Definition}
\newtheorem{remark}[lemma]{Remark}

\crefname{subsection}{Subsection}{Subsections}

\newcommand{\R}{\mathbb{R}}
\newcommand{\bbA}{\mathbb{A}}
\newcommand{\calA}{\mathcal{A}}

\newcommand{\C}{\mathbb{C}}
\newcommand{\Q}{\mathbb{Q}}
\newcommand{\Z}{\mathbb{Z}}

\newcommand{\N}{\mathbb{N}}

\newcommand{\B}{\mathcal{B}}

\renewcommand{\a}{\mathfrak{a}}

\newcommand{\frako}{\mathfrak{o}}
\newcommand{\frakq}{\mathfrak{q}}

\newcommand{\GL}{\mathrm{GL}}

\newcommand{\SL}{\mathrm{SL}}

\newcommand{\SO}{\mathrm{SO}}

\newcommand{\op}{\mathrm{op}}

\newcommand{\Eis}{\mathrm{Eis}}
 
\newcommand{\Ind}{\mathrm{Ind}}

\newcommand{\tr}{\mathrm{tr}}
\newcommand{\ind}{\mathrm{Ind}}
\newcommand{\supp}{\mathrm{supp}}

\newcommand{\level}{\mathrm{level}}
\newcommand{\triv}{\mathrm{triv}}

\newcommand{\calP}{\mathcal{P}}

\newcommand{\calB}{\mathcal{B}}
\newcommand{\calS}{\mathcal{S}}
\newcommand{\calL}{\mathcal{L}}

\newcommand{\calM}{\mathcal{M}}

\newcommand{\frakB}{\mathfrak{B}}

\newcommand{\Hom}{\mathrm{Hom}}

\newcommand{\disc}{\mathrm{disc}}
\newcommand{\One}{\mathds{1}}
\newcommand{\vol}{\mathrm{vol}}
\renewcommand{\d}{\,\mathrm{d}}

\usepackage{svg}
\usepackage[color,notref,notcite,final]{showkeys}
\renewcommand*{\showkeyslabelformat}[1]{%
   \fbox{\vbox{\hsize=1.1cm\normalfont\small\url{#1}\par}}}

\title{On the local $L^2$-Bound of the Eisenstein Series}

\author{Subhajit Jana}
\address{Queen Mary University of London, Mile End Road, London E14 NS, UK.}
\email{s.jana@qmul.ac.uk}

\author{Amitay Kamber}
\address{Centre for Mathematical Sciences, Wilberforce Road, Cambridge CB3 0WB, UK.}
\email{ak2356@dpmms.cam.ac.uk}

\begin{document}

\begin{abstract}
We study the growth of the local $L^2$-norms of the unitary Eisenstein series for reductive groups over number fields, in terms of their parameters. We derive a \emph{poly-logarithmic} bound on an average, for a large class of reductive groups. The method is based on Arthur's development of the spectral side of the trace formula, and ideas of Finis, Lapid, and M\"uller.

As applications of our method, we prove the optimal lifting property for $\mathrm{SL}_n(\mathbb{Z}/q\mathbb{Z})$ for square-free $q$, as well as the Sarnak--Xue \cite{sarnak1991bounds} counting property for the principal congruence subgroup of $\mathrm{SL}_n(\mathbb{Z})$ of square-free level. This makes the recent results of Assing--Blomer \cite{assing2022density} unconditional.
\end{abstract}

\maketitle

\section{Introduction}
Eisenstein series are ubiquitous in the theory of automorphic forms, in particular, they are crucially used by Langlands, and later by Arthur, in their seminal works on the development of the spectral decomposition of the space of automorphic forms of a reductive group and the trace formula;
see \cite{langlands1966eisenstein,langlands1989eisenstein,arthur2005intro,lapid2022perspective}. Many important questions in number theory and automorphic forms rely directly upon certain analytic properties of the Eisenstein series, e.g. their complex analytic properties including the location and order of their poles, and their various growth properties near the cusps.

The Eisenstein series, in contrast with \emph{cusp forms}, do not decay near the cusps; in fact, they have polynomial growth. The rate of growth depends precisely on the constant terms of the Eisenstein series. In particular, the unitary Eisenstein series (barely!) fail to be square-integrable on the non-compact fundamental domain. However, it is quite natural to wonder how the $L^2$-norms of unitary Eisenstein series grow \emph{locally} (i.e., on a fixed compact subset of the fundamental domain), in terms of their archimedean and non-archimedean parameters. 

In this paper, we study the general local $L^2$-bound of a unitary Eisenstein series on a reductive group over a number field. We refer to \cite[\S5.3]{lapid2022perspective} where such questions are considered and the relations with other important automorphic and analytic questions, e.g.\ location and order of the poles of the Eisenstein series and classification of the residual spectrum, are discussed. Informally, we prove \emph{sharp} upper bounds for the averages of the squared local $L^2$-norms of the Eisenstein series over certain \emph{short} families, for a large class of reductive algebraic groups; see \cref{thm:unramified} and \cref{thm:BCF}. We refer to \cref{local-l2-bound} for precise and detailed discussions of the problems and our results.

Our estimates of the local $L^2$-norms of the Eisenstein series are potentially useful to answer many interesting number-theoretic questions, in particular, the problems that are approached via the spectral theory of various \emph{non-compact} arithmetic symmetric spaces. Such approaches often face serious analytic difficulties due to the existence of the continuous spectrum and require estimates for various local $L^p$-norms of the Eisenstein series. Our estimates and the methods to prove them seem to be utilizable to overcome the analytic hurdles that arise via the local $L^2$-norms of the Eisenstein series, at least in certain cases. 

As a \emph{proof-of-concept}, we prove the \emph{optimal lifting property} for $\mathrm{SL}_n(\mathbb{Z}/q\mathbb{Z})$ in \cref{thm:optimal-lifting} for square-free $q$, as well as the \emph{Sarnak--Xue counting property} in \cref{thm:counting}, for the principal congruence subgroups of $\SL_n(\Z)$ of square-free level. Assing--Blomer recently proved those theorems \emph{conditional} on the local $L^2$-bounds of the Eisenstein series (see \cite{assing2022density}). Thus, we make their results unconditional. We refer to \cref{subsec:applications} for a detailed discussion. In a sequel \cite{jana2024optimal}, we prove \emph{optimal Diophantine exponent} for $\SL_n(\Z[1/p])$, as expected in a series of influential works by Ghosh--Gorodnik--Nevo (e.g. see \cite{ghosh2018best}), assuming \emph{Sarnak's density hypothesis}. The proof of the main results in \cite{jana2024optimal} crucially depends on (a slightly weaker version of) \cref{thm:gln-bcf} below; see \cite[\S 4.8]{jana2024optimal}. We refer the reader to \cite{jana2024optimal} for a detailed discussion of this problem and our proof.

\vspace{0.3cm}

Let us now discuss the motivating questions and our results.

\subsection{Local \texorpdfstring{$L^2$}{L2}-bound}\label{local-l2-bound}
Let $G$ be a reductive algebraic group over a number field $F$. Let $\bbA$ be the adele ring of $F$. Also let $\bbA_f$ be the finite adeles and $F_\infty := F\otimes\R$. Given a Levi subgroup $M$ of $G$, let $\a_M^*$ be the $\R$-vector space spanned by $F$-rational characters of $M$, and $(\a_M^G)^*$ the $\R$-vector space spanned by characters that are trivial on $G$. Let $H_M\colon M(\bbA) \to \a_M$ the natural map, and $M(\bbA)^1$ the kernel of $H_M$. We denote by $M_0$ a minimal Levi subgroup.

Let $P$ be a parabolic subgroup of $G$ with Levi subgroup $M$. We denote $\a_P = \a_M$. Let $\varphi$ be in $\calA^2(P)$ which is the pre-Hilbert space of square-integrable automorphic forms on $N(\bbA)M(F)\backslash G(\bbA)$. Let $\lambda$ be an element of $(\a_{M,\C}^G)^* := (\a_{M}^G)^*\otimes_\R \C$.
Following Langlands in \cite{langlands1989eisenstein} (also see \cite[Chapter 7]{arthur2005intro}) we construct an Eisenstein series $\Eis(\varphi,\lambda)$ attached to $\varphi$ and $\lambda$; see the details of the construction in \cref{sec:eisenstein-series}.

Let $\varphi$ be a \emph{unit} vector (see \cref{sec:automorphic-forms} for the definition of the norm on $\calA^2(P)$) with $\nu(\varphi)$ being a measurement of the complexity of $\varphi$. We will later use a more nuanced notion of complexity (both archimedean and non-archimedean); see \cref{sec:complexity} for precise definitions. Finally, let $\lambda \in i(\a_M^G)^*$ be purely imaginary, so that $\Eis(\varphi,\lambda)$ appears in the spectral decomposition of $L^2(G(F)\backslash G(\bbA)^1)$.  Here and elsewhere in the paper $\d\lambda$ will denote the real-valued Lebesgue measure on $i(\a_M^G)^*$ as in e.g.\ \cite{arthur2005intro}.

We ask the following two questions with (presumably, strictly) decreasing levels of difficulty. Let $\Omega$ be a fixed compact set in a fundamental domain of $G(F)\backslash G(\bbA)^1$. How do
\begin{equation}\label{pointwise-bound}
    \intop_\Omega|\Eis(\varphi,\lambda)(g)|^2\d g
\end{equation}
and
\begin{equation}\label{average-bound}
    \intop_{\|\lambda'-\lambda\|\le 1}\intop_\Omega|\Eis(\varphi,\lambda')(g)|^2\d g\d\lambda'
\end{equation}
grow as $\nu(\varphi)+\|\lambda\|\to\infty$?

\vspace{0.3cm}

Here is our general expectation regarding the estimates of the above integrals.
\begin{equation}\label{eq:expectation}
    \text{The integrals }\eqref{pointwise-bound}\text{ and }\eqref{average-bound} \ll_{\Omega} \left(\log(1+\nu(\varphi)+\|\lambda\|)\right)^{\dim\a_M^G},
\end{equation}
as $\nu(\varphi)+\|\lambda\|\to\infty$.

\vspace{0.3cm}

Let $\varphi$ be spherical at all places and $\nu(\varphi)$ denote the size of the Laplace eigenvalue of $\varphi$. If $G=\GL_2$ and $\Eis(\varphi,\lambda)$ is the non-holomorphic spherical Eisenstein series then it is a classical result that \eqref{pointwise-bound} indeed satisfies \eqref{eq:expectation}; see \cref{sec:difficulties}.
For $G=\GL_3$, Miller in \cite{miller2001existence} proved that \eqref{average-bound} satisfies \eqref{eq:expectation}; see \cref{sec:difficulties} for more discussions. For $G=\GL_n$, Zhang in \cite{zhang2019quantum} showed that if $\Eis(\varphi,\lambda)$ is the maximal parabolic \emph{degenerate} Eisenstein series, i.e.\ induced from the trivial representation of the Levi subgroup $\GL_{n-1}\times\GL_1$, then \eqref{pointwise-bound} satisfies \eqref{eq:expectation}\footnote{In fact, he proved an asymptotic formula in $\lambda$ in this case.}. As far as we know, these are the only examples where a poly-logarithmic strength bound is known for all the Eisenstein series for either \eqref{pointwise-bound} or \eqref{average-bound}, at least for higher rank groups.

\begin{remark}
In this remark we comment on a few topics in the existing literature that have similar flavour as the local $L^2$-bound question in this paper. One may wonder about the size of the analogous local (and global, for cuspidal forms) $L^p$-bound of automorphic forms for any $1\le p\le \infty$. For compact manifolds, general polynomial bounds in $\nu(\varphi)$ were obtained by Sogge \cite{sogge1988lp-norm} and there have been many improvements in various aspects. On arithmetic locally symmetric spaces local $L^\infty$-norms for cuspidal automorphic forms have been studied quite a lot, in particular, quite extensively for small rank groups. We refer to \cite{blomer2016supnorm-pgln} for the case of $\GL_n(\R)$, to \cite{marshall2014upper} for general semisimple groups and the references therein. The proofs there are via the pre-trace formula which can also be made to work to prove similar bounds for the local $L^\infty$-bounds of the short averages of the square of the Eisenstein series in case of non-compact spaces. We also remark that local $L^\infty$-bound of general Eisenstein series may not satisfy \eqref{eq:expectation}; see, e.g., \cite[Theorem 2]{blomer2020epstein}. On the other hand, the local $L^\infty$ bounds of the cusp forms (similarly, of an average of Eisenstein series) on $\GL_n$ as in \cite{blomer2016supnorm-pgln} often do not hold globally (i.e.\ uniformly near the cusp); see \cite{brumley2020large}.

\end{remark}

\subsection{Main results}

In this section we state our main results. Informally, we prove that, for a large class of groups, \eqref{average-bound} indeed can be estimated \emph{poly-logarithmically} in $\nu(\varphi)+\|\lambda\|$; as expected in \eqref{eq:expectation}. Before stating the most general result in our paper, we first give the \emph{state-of-the-art} result for $\GL_n$ that is possible to state in classical language and without detailed notations.

\begin{theorem}\label{thm:gln-bcf}
Let $G= \GL_n$ and $P$ be a parabolic subgroup of $G$. Let $\varphi_0\in\calA^2(P)$ be a unit that is spherical at all places. Moreover, suppose that $\varphi_0$ is an eigenfunction of the full Hecke algebra at the non-archimedean places and the ring of invariant differential operators in the center of the universal enveloping algebra at the archimedean places. Then for any $\lambda_0\in i(\a_P^G)^*$ we have
\[
\intop_{\substack{{\lambda\in i(\a_P^G)^*} \\ \|\lambda-\lambda_0\|\le 1}}\intop_\Omega|\Eis(\varphi_0,\lambda)(x)|^2\d x\d \lambda \ll_\Omega \left(\log(1+\nu(\varphi_0)+\|\lambda_0\|)\right)^{n-1},
\]
where $\nu(\varphi_0)$ is the size of the Laplace eigenvalue of $\varphi_0$.
\end{theorem}

\begin{remark}\label{rem:level-gln}
Actually, \cref{thm:gln-bcf} is true for any $\varphi$ that lies in the $\pi$-isotypic subspace $\calA_\pi^2(P)$ for some representation $\pi$ of $M(\bbA)$, and is only required to be spherical at the archimedean places and invariant by any $K\subset\GL_n(\bbA_f)$, with $\ll$ replaced by $\ll_K$. The dependency of the constant on $K$ could, a priori, be large, e.g. polynomial in $\level(K)$ (see \cref{sec:complexity} for the definition of level). One may expect a polylogarithmic dependence on the level, but we can only prove it when $\varphi$ is associated with \emph{unramified cuspidal datum}, as in \cref{thm:unramified}. 
\end{remark}

\vspace{0.5cm}

Now we describe our full results. First, we recall certain terminologies (that are fully explained in the body of the paper).

\begin{itemize}
    \item Following Langlands, a \emph{cuspidal datum} $\chi$ is a $G(F)$-conjugacy class of pairs $(L,\sigma)$, where $L$ is a Levi subgroup of $G$ defined over $F$ and $\sigma$ is a cuspidal representation of $L(\bbA)$; see \cref{sec:automorphic-forms}. Moreover, we recall the notion of \emph{unramifiedness} of a cuspidal datum, due to Arthur, which ensures that the datum has a \emph{trivial} stabilizer; see \cite[\S 15]{arthur2005intro}\footnote{The notion should not be confused with ``unramifiedness'' of the local components of an automorphic representation.}.
    \item We discuss certain estimates of various intertwining operators. The following terminologies are due to Finis--Lapid--M\"uller.
    \begin{itemize}
        \item We say $G$ satisfies \emph{property (TWN+)} if the first derivatives of the global normalization factors of the intertwining operators satisfy certain estimates; see \cref{sec:twn+}.
        \item We say $G$ satisfies \emph{property (BD)} if the first derivatives of the local normalized intertwining operators satisfy certain estimates; see \cref{sec:bd}.
    \end{itemize}
    \item For any Levi subgroup $M$ let $\Pi_2(M)$ denote the isomorphism class of representations of $M(\bbA)$ appearing discretely in $L^2(M(F) \backslash M(\bbA))$; see \cref{sec:automorphic-forms}.
    \item For a parabolic subgroup $P$ with Levi subgroup $M$, $\chi$ a cuspidal datum, and $\pi \in \Pi_2(M)$ let $\calA^2(P)_{\chi,\pi}$ denotes the $(\chi,\pi)$-isotypic part of $\calA^2(P)$; see \cref{sec:automorphic-forms}.
    \item We fix (good) maximal compact subgroups $K_\infty$ of $G(F_\infty)$ and $K_f$ of $G(\bbA_f)$. For $\tau\in \widehat{K_\infty}$, which is the unitary dual of $K_\infty$, we write $\calA^2(P)^{\tau,K}$ for the space of functions in $\calA^2(P)$ which are right $K$-invariant and belong to $\tau$-isotypic subspace. We define $\calA^2_{\chi,\pi}(P)^{\tau,K}$ similarly; see \cref{sec:automorphic-forms}.
    \item We bound the integrals over $\Omega$ in \eqref{pointwise-bound} and \eqref{average-bound} via the inner product of truncated Eisenstein series. We work with a \emph{truncation operator} $\Lambda^T$, as defined by Arthur, where $T$ is a sufficiently dominant element in $\a_{M_0}$. We also define $d(T)$ to be a measurement of how \emph{dominant} $T$ is; see \cref{sec:truncation} for details.
    \item Finally, we fix certain measurements of the {archimedean complexity} of a representation $\pi$ of a real group by the size of its \emph{Casimir eigenvalue}, denoted by $\nu(\pi)$ and certain measurement of the {non-archimedean complexity} of an open compact subgroup $K$ of $G(\bbA_f)$ by its \emph{depth and level}, denoted by $\level(K)$; see \cref{sec:complexity}. To give a sense of their sizes for $G=\GL_2$, if $\pi$ is principal series representations induced from $|.|^{it}\otimes|.|^{-it}$ then $\nu(\pi)\asymp 1+t^2$, and if $K$ is the principal congruence subgroup of level $q$ then $\level(K)=q$.
\end{itemize}

\vspace{0.3cm}

Here is our first main theorem.
\begin{theorem}\label{thm:unramified}
Assume that $G$ satisfies properties (TWN+) and (BD). Let $P$ be a parabolic subgroup of $G$ with Levi subgroup $M$ and $\pi_0\in\Pi_2(M)$. Moreover, assume that $\chi:=(M,\pi_0)$ is an \emph{unramified cuspidal datum} of $G$. Let $\tau\in\widehat{K_\infty}$ and $K\subset K_f$ be an open-compact. Then
\[
\intop_{\substack{{\lambda\in i(\a_M^G)^*} \\ \|\lambda-\lambda_0\|\le 1}} \|\Lambda^T \Eis(\varphi_0,\lambda)\|_{2}^2 \d \lambda \ll \Big((1+\|T\|)\,\log \left(1+\nu(\pi_{0,\infty})+\|\lambda_0\|+\nu(\tau)+\level(K)\right)\Big)^{\dim\a_M^G},
\]
for any unit $\varphi_0\in \calA^2_{\chi,\pi_0}(P)^{\tau,K}$ and $\lambda_0\in i(\a_M^G)^*$.
\end{theorem}

\begin{remark}
The implicit constant in \cref{thm:unramified} depends only on $G$. For $G$ and $\chi$ as in the statement, \cref{thm:unramified} proves that the integral \eqref{average-bound} grows poly-logarithmically in all the parameters, verifying \eqref{eq:expectation}. For $\GL_n$, when \cref{thm:unramified} applies, it is more general than \cref{thm:gln-bcf} in the level and weight aspect, since it gives poly-logarithmic dependence on the level and is not restricted to $K_\infty$-invariant vectors. 
\end{remark}

\vspace{0.5cm}

Of course, \cref{thm:unramified} excludes a ``narrow'' class of cuspidal data that are \emph{not} unramified. For example, for $G=\GL_n$ a Levi subgroup is of the form $M\cong \GL_{n_1}\times\dots\times \GL_{n_k}$ and a representation $\pi \in \Pi_2(M)$ satisfies $\pi \cong \pi_1\otimes\dots\otimes \pi_k$, where $\pi_i \in \Pi_2(\GL_{n_i})$. \cref{thm:unramified} will apply to $\pi$ if and only if each $\pi_i$ is cuspidal, and $\pi_i \not \cong \pi_j$ for $i\ne j$. 

In \cref{thm:BCF} we give a result towards the expectation in \eqref{eq:expectation} for $G$ with the similar level of generalities considered in \cref{thm:unramified} and $\tau$ being trivial, but for \emph{all} cuspidal data $\chi$. Before stating the theorem precisely we introduce a new quantity below, which roughly measures the number of discrete series representations with specific $K_\infty K_f$-type that arises from a given cuspidal datum.

\begin{defn}\label{BCF-def}
Let $\tau\in\widehat{K_\infty}$ and $K\subset K_f$ be an open-compact. 
For any cuspidal datum $\chi$ of $G$ we define
\[
F(\chi;\tau,K):=\sum_{P\supset P_0}\dim\left(\calA^2_{\chi}(P)^{\tau,K}\right)=\sum_{P\supset P_0}\sum_{\pi\in \Pi_2(M_P)}\dim\left(\calA^2_{\chi,\pi}(P)^{\tau,K}\right).
\]
Here $P_0$ is some fixed minimal parabolic subgroup containing $M_0$ and $P\supset P_0$ are called \emph{standard parabolic} subgroups.
\end{defn}

\begin{theorem}\label{thm:BCF}
Assume that $G$ satisfies properties (TWN+) and (BD). 
Let $\chi:=(M,\pi_0)$ be any cuspidal datum of $G$.
Then 
\begin{align*}
\sum_{P\supset P_0}&\sum_{\pi\in\Pi_2(M_P)}\sum_{\varphi\in\calB_{\chi,\pi}(P)^{K_\infty K}}\intop_{\substack{{\lambda\in i(\a_P^G)^*} \\ \|\lambda-\lambda_0\|\le 1}} \|\Lambda^T \Eis(\varphi,\lambda) \|^2_2 \d \lambda \\ 
\ll & \Big((1+\|T\|)\,\log(1+\nu(\pi_{0,\infty})+\|\lambda_0\|+\level(K))\Big)^{\dim\a_M^G}F(\chi;\triv,K),
\end{align*}
where $\calB_{\chi,\pi}(P)^{K_\infty K}$ denotes an orthonormal basis of $\calA^2_{\chi,\pi}(P)^{K_\infty K}$ and $\lambda_0\in i(\a_M^G)^*$.
\end{theorem}

\begin{remark}
    Note that the $(P,\pi)$-summand on the left-hand side above vanishes identically unless $\calA^2_{\chi,\pi}(P)^{K_\infty K}\neq\{0\}$. In fact, the total number of nonzero $(P,\pi,\varphi)$-summands on the left-hand side is finite and equals to $F(\chi;\triv,K)$ (the $\lambda$-integral however may be empty if $\lambda_0 \in i(\a_M^G)^* \subset i(\a_{M_0}^G)^*$ is far from $i(\a_{P}^G)^* \subset i(\a_{M_0}^G)^*$). 
    Thus \cref{thm:BCF} \emph{on average} is of the same strength as \cref{thm:unramified}, that is, the integral in \eqref{average-bound} \emph{on average} satisfies the estimate in \eqref{eq:expectation}.
    
\end{remark}

In particular, by the positivity of the left-hand side in \cref{thm:BCF} we obtain the following corollary.

\begin{cor}\label{cor:BCF}
Assume that $G$ and $\chi$ are as in \cref{thm:BCF}. 
Let $P$ be a parabolic subgroup of $G$ and $\pi\in\Pi_2(M_P)$ such that $\calA^2_{\chi,\pi}(P)\neq\{0\}$. 
Then
\begin{align*}
\intop_{\substack{{\lambda\in i(\a_P^G)^*} \\ \|\lambda-\lambda_0\|\le 1}} &\|\Lambda^T \Eis(\varphi_0,\lambda) \|^2_2 \d \lambda \\ 
\ll & \Big((1+\|T\|)\,\log(1+\nu(\pi_{\infty})+\|\lambda_0\|+\level(K))\Big)^{\dim\a_M^G}F(\chi;\triv,K),
\end{align*}
for any unit $\varphi_0\in \calA^2_{\chi,\pi}(P)^{K_\infty K}$ and $\lambda_0\in i(\a_{P}^G)^*$.
\end{cor}

\begin{remark}\label{rem:thm3 problems}

Let us focus on $\GL_n$. In this case, understanding $F(\chi;\triv,K)$ boils down to a local question, as we know for which $\pi$ it holds that $\calA^2_{\chi,\pi}\ne\{0\}$ thanks to \cite{moeglin1989residual}. Then one can show that when $\pi_0$ is cuspidal it holds that 
\[
F(\chi;\triv,K) \ll \dim \calA^2_{\chi,\pi_0}(P)^{K_\infty K}.
\]
However, when $\pi_0$ is not cuspidal $F(\chi;\triv,K)$ can be a lot larger than $\dim \calA^2_{\chi,\pi_0}(P)^{K_\infty K}$, which eventually leads to a non-optimal result, when averaging over a basis of $\calA^2_{\chi,\pi_0}(P)^{K_\infty K}$.

For general groups, the size of $F(\chi;\tau,K)$ has two components. The local component, namely, the dimension of the $K_\infty K_f$-type subspace and the global component, namely, the number and multiplicity of square-integrable representations (with a given $K_\infty K_f\cap M(\bbA)$-type) that are attached to $\chi$. 

The local component could, potentially, be quite larger than $\log(\|\nu(\tau)\|+\level(K))$, but by uniform admissibility (see \cite{bernstein1974reductive}) is bounded depending only on $\tau$ and $K$. This non-optimal dependence on $\level(K)$ even for $\GL_n$ is the main weakness of \cref{thm:BCF} relative to \cref{thm:unramified}. Also, \cref{thm:BCF} is restricted to $\tau=\triv$, but this should not be a major issue.

The global component should, at least conjecturally, be small, and in fact bounded (perhaps depending on $K$). Its understanding depends eventually upon the classification of the residual spectrum. In other words, it depends on how the residual spectrum arises from the residues of the Eisenstein series induced from cuspidal representations, which in turn asks for highly non-trivial analytic properties of the Eisenstein series. All of the above problems are still wide open for general reductive groups, and understanding them for classical groups seems like a very interesting question that we plan to address in future work.
\end{remark}

\begin{remark}
One may wonder whether one can obtain an estimate for the pointwise bound of the local $L^2$-norms in \eqref{pointwise-bound} from corresponding bounds for their $\lambda$-averages in \cref{cor:BCF}. In principle, one may attempt this via a contour-shifting argument and employing the functional equation of the Eisenstein series. This kind of approach is prevalent in analytic number theory: for example, bounding central $L$-value $L(1/2+iT)$ for some automorphic representation for large $T$, by sharp estimation of $\int_{T-H}^{T+H}|L(1/2+it)|^2\d t$ for some $H\le T$, shifting the $it$-contour, and applying functional equation of the $L$-function; see e.g. \cite{good1982square}. 

Clearly, to apply the above contour-shifting method in our case we need to crucially understand meromorphic properties of the Eisenstein series \emph{away from the unitary axis}, in particular, the location of its poles and residues. On $\GL_n$, this is essentially equivalent to understanding \emph{strong} zero-free regions of various Rankin--Selberg $L$-functions. Our current knowledge of such zero-free regions is rather limited and only yields a polynomial strength bound in the question \eqref{pointwise-bound}.
\end{remark}

\vspace{0.5cm}

\subsection{Applications}\label{subsec:applications}

While estimating the local $L^2$-bound of the Eisenstein series is interesting in its own right, it is also very useful in applications to number theory and Diophantine approximation. As a matter of fact, we came to study $L^2$-bounds of the Eisenstein series while trying to understand certain number theoretic problems via the spectral theory of automorphic forms. 

We already mentioned the sequel \cite{jana2024optimal}, where we use the results of this work to prove \emph{optimal Diophantine exponent} for $\SL_n$ assuming \emph{Sarnak's density hypothesis}.

Here we give two other applications of the $L^2$-bounds of the Eisenstein series, concerning the principal congruence subgroups of $\SL_n(\Z)$. Let $\Gamma(1) := \SL_n(\Z)$ and let $\Gamma(q) := \ker(\Gamma(1) \xrightarrow{\mod q} \SL_n(\Z/q\Z))$ be its principal congruence subgroup. By strong approximation, $\Gamma(1)/ \Gamma(q) \cong \SL_n(\Z/q\Z)$. Let $\|\cdot\|$ be the Frobenius norm on $\SL_n(\R)$, i.e. $\|g\|^2 = \sum_{i,j} |g_{i,j}|^2$ (as a matter of fact, one can consider any other norm as well).

The first problem is about \emph{counting} the number of elements 
\[\# \{\gamma \in \Gamma(q) \mid \|\gamma\| \le R\}.\]
There are two regimes in which this problem is well understood. In the regime $R< Cq$ where $C=C(n)$ is some small constant, it is simple to show that 
\[\# \{\gamma \in \Gamma(q) \mid \|\gamma\| \le R\} = 1.\]
In the regime where $q$ is fixed and $R\to \infty$ (or at least a large power of $q$), it holds that 
\begin{equation}\label{eq:DRS}
\# \{\gamma \in \Gamma(q) \mid \|\gamma\| \le R\} = c_n\frac{R^{n(n-1)}}{\#\SL_n(\Z/q\Z)}+ O(R^{n(n-1)-\delta_n})     
\end{equation}
with $c_n$ an explicit constant, and $\delta_n>0$ a small constant. This follows from the work of Duke, Rudnick and Sarnak \cite[Theorem~1.10]{duke1993density}, and is based on spectral methods\footnote{The result is for $q=1$, but the method works for every $q$. See also \cite[Corollary 5.2]{gorodnik2015quantitative} for a generalized statement with a simpler proof. Recently, in \cite{blomer2023hyperbolic}, the authors significantly improved the error term of \cref{eq:DRS} over \cite{duke1993density}. 
}. 

Sarnak and Xue (\cite[Conjecture 2]{sarnak1991bounds}, see also \cite{sarnak1991diophantine}) made the following conjecture, which addresses the intermediate range when $R$ is a small power of $q$:
\[
\# \{\gamma \in \Gamma(q) \mid \|\gamma\| \le R\} \ll_\epsilon \frac{R^{n(n-1)+\epsilon}}{|\SL_n(\Z/q\Z)|}+ R^{n(n-1)/2} .
\]
In a recent work, Assing and Blomer \cite[Theorem~1.4]{assing2022density} essentially confirm the above Sarnak--Xue conjecture showing that
\begin{equation}\label{eq:counting}
\# \{\gamma \in \Gamma(q) \mid \|\gamma\| \le R\}
\ll_\epsilon (Rq)^\epsilon\left(\frac{R^{n(n-1)}}{q^{n^2-1}}+R^{n(n-1)/2}\right),
\end{equation}
\emph{conditional} on a difficult conjecture, namely \cite[Hypothesis 1]{assing2022density}, on the local $L^2$-bound of the Eisenstein series. This hypothesis is slightly weaker than the expected bound in \eqref{eq:expectation} for the integral \eqref{pointwise-bound}. In \cref{thm:counting}, we make Assing--Blomer's result \emph{unconditional}, by employing the method of proving \cref{thm:BCF}, while also using their work as a crucial input.

\begin{theorem}\label{thm:counting}
Let $q$ be a square-free integer. Then \cref{eq:counting} holds
for every $\epsilon>0$ and every $R\ge 1$.
\end{theorem}

\vspace{0.3cm}

The second application we have is the problem of \emph{optimal lifting}. Recall that by strong approximation for $\SL_n$, the map $\SL_n(\Z) \to \SL_n(\Z/q\Z)$ is onto. We may therefore ask, given $\overline{\gamma} \in \SL_n(\Z / q\Z)$, to find a lift of it to $\gamma \in \SL_n(\Z)$, with $\|\gamma\|$ minimal. Gorodnik--Nevo studied this problem in a general context and proved that for every $\overline{\gamma}$ it is possible to find a lift $\gamma$ with $\|\gamma\| \ll_\epsilon q^{C_n+\epsilon}$ where $C_n$ is some explicit constant; see \cite[Theorem 1.1]{gorodnik2012lifting}. One may then wonder what is the smallest $C_n$ so that the statement holds for it. A simple lower bound is $C_n \ge 1+1/n$, which follows from \cref{eq:DRS} (with $q=1$) and the fact that $|\SL_n(\Z/q\Z)| \gg_\epsilon q^{n^2-1-\epsilon}$, but it is not the correct answer. In \cite{kamber2023lifting} by the second author and P{\'e}t{\'e}r Varj{\'u} (which was completed after this work), it is shown that the correct answer is $C_n= 2$ (following weaker results of \cite{sarnak2015lettermiller}). More precisely, the authors showed that for every $q$ and $\overline{\gamma} \in \SL_n(\Z / q\Z)$, there is a lift $\gamma \in \SL_n(\Z)$ with $\|\gamma\| \ll_n q^2 \log q$. In the other direction, they show that for every $n$, $\epsilon>0$, and for every $q$ there is $\overline{\gamma} \in \SL_n(\Z / q\Z)$ such that every lift $\gamma \in \SL_n(\Z)$ of $\overline{\gamma}$ satisfies $\|\gamma\| \gg_{n,\epsilon} q^{2-\epsilon}$.

In contrast to the results of \cite{kamber2023lifting}, Sarnak \cite{sarnak2015lettermiller} observed that when one considers \emph{almost all} the elements instead of \emph{all} the elements, the simple lower bound $C_n=1+1/n$ should be the correct answer to the problem, and proved it for $n=2$. This property is expected to hold in a general context and we call it \emph{optimal lifting}. We again refer to \cite{golubev2020sarnak} for a detailed discussion of the general problem of optimal lifting, and to \cite{kamber2019optimalSL3} for this specific problem. 

In the recent work of Assing and Blomer \cite[Theorem~1.4]{assing2022density}, they essentially solved the optimal lifting problem (see the statement of \cref{thm:optimal-lifting}), but again \emph{conditional} on \cite[Hypothesis 1]{assing2022density}. Once again, in \cref{thm:optimal-lifting}, employing the method of proving \cref{thm:BCF} and crucially inputting one of the main results of \cite{assing2022density}, we make the above result of Assing--Blomer's \emph{unconditional}.

\begin{theorem}\label{thm:optimal-lifting}
For every $\epsilon>0$ there is $\delta>0$ such that for every square-free integer $q$, the number of elements in $\overline{\gamma}\in\SL_n(\Z/q\Z)$ without a lift $\gamma\in \SL_n(\Z)$ with $\|\gamma\| \le q^{1+1/n+\epsilon}$ is at most $O(q^{n^2-1-\delta})$.
\end{theorem}

We now add a few remarks regarding \cref{thm:counting} and \cref{thm:optimal-lifting}.

\begin{remark}
While \cref{thm:counting} and \cref{thm:optimal-lifting} are limited to $q$ square-free (except for $n=2$, where Sarnak's proof holds for all $q$), our method can prove them for general $q$ assuming Sarnak's density conjecture, which is implied by the Generalized Ramanujan Conjecture (GRC) for cusp forms of $\GL_n$ over $\Q$. 
More generally, unconditionally and independently of (most of) the work of Assing--Blomer, one can get slightly weaker results, with the exponents multiplied by $1+2\delta/(n-1)$, where $0 \le \delta \le (n-1)/2$ is the best approximation for the GRC for $\GL_n$ over $\Q$. Specifically, one has $\delta \le \frac{1}{2}-\frac{1}{1+n^2}$; see \cite{sarnak2005notes}.
\end{remark}

\begin{remark}
Conjecturally, similar ``Sarnak--Xue Counting'' and ``Optimal Lifting'' theorems should hold more generally, in particular, for congruence subgroups of arithmetic lattices in semisimple Lie groups. We refer to \cite{golubev2020sarnak}, where those conjectures are discussed at length (for the specific problem of $\Gamma(q)$ see \cite[\S 2.6]{golubev2020sarnak}). In particular, for \emph{uniform (cocompact) lattices} one can prove that Sarnak--Xue counting and optimal lifting follow from \emph{Sarnak's density conjecture} in automorphic forms, which is, presumably, a true approximation of the (false) \emph{Naive Ramanujan Conjecture} (see \cite{golubev2020sarnak}). For non-uniform lattices, this implication is unknown, because of the analytic difficulties coming from the Eisenstein series. The techniques of this paper allow us to essentially overcome those analytic difficulties, for non-uniform lattices in groups different than $\GL_n$, as long as properties (TWN+) and (BD) are known. 
For principal congruence subgroups of $\SL_n(\Z)$ of square-free level Assing--Blomer proved Sarnak's density conjecture, which allows in combination with this work to deduce the two theorems.
\end{remark}

\begin{remark}
    We note that in this paper we do not attempt to prove \cite[Hypothesis 1]{assing2022density}, and it remains an interesting and hard open question. We refer to \cref{subsec:proof of application} for a discussion.
\end{remark}

\subsection{Difficulties in higher rank}\label{sec:difficulties}

In this section, we try to describe the difficulties one faces while trying to estimate \eqref{pointwise-bound} and \eqref{average-bound}.

We start with a toy-example -- estimating \eqref{pointwise-bound} for a low-rank group. We take $G=\GL_2$ over $\Q$, $M=\GL_1\times\GL_1$, $P$ the Borel subgroup in $G$, $\pi=\triv\times\triv$, and $\lambda=(it,-it)$ with $t\in\R$. We take $\varphi\in\calA^2(P)_\pi$ to be the $\mathrm{O}_2(\R)\GL_2(\hat{\Z})$-invariant unit vector. In the classical language $\Eis(\varphi,\lambda)$ is the same as $E(\cdot,1/2+it)$. The classical Maass--Selberg relations (see \cite[pp.153]{miller2001existence}) yield that
\begin{align*}
    &\|\Lambda^TE(\cdot,1/2+it)\|^2_2 
    \,=\lim_{t'\to t}\langle\Lambda^TE(\cdot,1/2+it'),\Lambda^TE(\cdot,1/2+it)\rangle\\
    &=\lim_{t'\to t}\left(\frac{e^{iT(t'-t)}}{i(t'-t)}+\frac{\overline{c(it')}e^{iT(t'+t)}}{i(t'+t)}-\frac{c(it)e^{-iT(t'+t)}}{i(t'+t)}-\frac{{c(it')}\overline{c(it)}e^{-iT(t'-t)}}{i(t'-t)}\right)\\
    &=2T-\frac{c'(it)}{c(it)} + \Re\left(\frac{c(it)e^{-2itT}}{t}\right)\ll T\log(1+|t|).
\end{align*}
Here $c(s)$ is the \emph{scattering matrix} appearing in the constant term of $E(\cdot,1/2+it)$, given by $\frac{\xi(s)}{\xi(1+s)}$, where $\xi(s)$ is the completed Riemann zeta function. The last estimate follows from the standard zero-free region of the Riemann zeta function.

Now we explain the difficulties we face if we try to prove the statement given in \eqref{eq:expectation}, by generalizing the method above.
\begin{enumerate}
    \item The case when $\Eis(\varphi,\lambda)$ is the minimal parabolic Eisenstein series for $\GL_n$ is a natural higher-rank analogue of the above. Higher rank Maass--Selberg relations are in this case more complicated than the $\GL_2$ version. The number of summands in the formula is ($n!)^2$. Moreover, each summand will likely have (a priori, higher-order) poles as $\lambda'\to\lambda$. The maneuver of ``grouping terms'' which cancel each other's poles, analogous to the $\GL_2$ case, is in fact, combinatorially quite involved. Such complications arise already for $\GL_3$, as can be seen in \cite{miller2001existence}.
    
    \item Apparently, the next difficult case is when $\Eis(\varphi,\lambda)$ is induced from a cuspidal representation of a non-minimal parabolic, i.e., $\varphi$ is a cusp form. Apart from solving the combinatorial problem that arises in the limit procedure, one also needs to know certain strong estimates of the logarithmic (a priori, high) derivatives of the intertwining operators. When the Eisenstein series is spherical everywhere such estimates can be described via the zero-free region of various Rankin--Selberg $L$-functions for $\GL_{m_1}\times\GL_{m_2}$ with $m_1+m_2\le n$ at the edge of the critical strips. However, the quality and generality of the zero-free region that is required to prove \eqref{eq:expectation} is still out of reach.
    
    \item We remark that Langlands (see \cite[Proposition 15.3]{arthur2005intro}) proved the Maass--Selberg relations, but only when the underlying representation $\pi$ of $M(\bbA)$ is \emph{cuspidal}. Arthur, in \cite{arthur1982inner}, proved an approximate Maass--Selberg relations for a general discrete series $\pi$. However, the precision and generality of the approximation that we need to prove \eqref{eq:expectation} is not available. In fact, the required precision in Arthur's approximate Maass--Selberg formula can be obtained only via the classification of discrete series representations, in particular, how the residual spectrum arises from the Eisenstein series induced from cusp forms. This, in turn, requires understanding deep meromorphic properties of the Eisenstein series, in particular, the locations, orders, and residues of their poles.
    
    \item When we move from $\GL_n$ to general reductive groups, all of the above difficulties persist prominently. First of all, our ignorance of the classification of the discrete series, equivalently the required meromorphic properties of a general Eisenstein series, hinders the very first step of estimating the $L^2$-norm via the Maass--Selberg relations. Even in the case of the Eisenstein series induced from cuspidal representations, the limit procedure requires a deep understanding of the meromorphic properties of the intertwining operators and the size of the (a priori, high) derivatives of the same - none of which is yet in the grasp of the current literature.
\end{enumerate}

Primarily, this is why in this paper we try to estimate \eqref{average-bound} for general Eisenstein series on general reductive groups \emph{not} via the Maass--Selberg formula, rather via the spectral side of the Arthur--Selberg trace formula. We explain this in the next subsection.

In a future work, in order to estimate \eqref{pointwise-bound} for $\GL_n$ using the Maass--Selberg relations, we plan to address some of the issues pointed above and in particular, to give a robust method of understanding the combinatorics and corresponding limit problem mentioned above.

\begin{remark}
    There is one case in which the Maass--Selberg relations do yield a nice formula, and it is exactly the unramified case of \cref{thm:unramified}. The Maass--Selberg relations in this case are given in \cref{lem:Maass-Selberg unram}, and essentially reduce to understanding a certain operator we denote by $\calM_M^T(\lambda,P)$, defined by a certain limit. Arthur's method of \emph{$(G,M)$-family} gives a way to calculate the limit in the formula. More precisely, Finis--Lapid--M\"uller \cite{finis2011spectral-abs}, generalizing previous results of Arthur in the spherical case, essentially gave a formula for $\calM_M^T(\lambda,P)$, given in \cref{prop:FLM formula}. The formula is a polynomial in $T$ with coefficients given by certain intertwining operators and their \emph{first} derivatives, given along \emph{linearly independent directions}. Then when we take $\lambda$-average over a short interval around $\lambda_0$, we can use properties (TWN+) (and (BD) for general congruence subgroups) to deduce a bound on $\calM_M^T(\lambda,P)$, given in \cref{prop:calM bound}, from which \cref{thm:unramified} follows. This technique can be traced back at least to \cite{muller2007weyl} (for $\GL_n$ and a fixed level), and was formalized in \cite{finis2015limit}.
\end{remark}


\subsection{High level sketch for the proof of \texorpdfstring{\cref{thm:BCF}}{Theorem 3}}

Our proof  of \cref{thm:BCF} uses the Arthur--Selberg trace formula, or more precisely, Arthur's development of the spectral side of the trace formula and Finis--Lapid--M\"uller's refinement of it.

As we described in the previous subsection, it is quite difficult to estimate $\|\Lambda^T \Eis(\varphi_0,\lambda_0)\|^2_2$ pointwise in $\lambda_0$. Proving an estimate on a $\lambda$-average over a short interval around $\lambda_0$ seems to be more tractable. Moreover, in our proof we also need to average over $\varphi$ in an orthonormal basis, equivalence classes of associated parabolic subgroups, and discrete series representations of the Levi of these parabolic subgroups.

For a sufficiently dominant truncation parameter $T$ and a bi-$K_\infty K$-invariant test function $f$ on $G(\bbA)$, Arthur defined a distribution $J_\chi^T(f)$ on the spectral side, that (for $d(T)$ sufficiently large) essentially can be given by a weighted sum over $P\supset P_0$ of
\[\sum_{\pi\in\Pi_2(M_P)}\sum_{\varphi\in\B_{\chi,\pi}(P)}\int_{i(\a_M^G)^*}\tilde{f}(\mu_\pi+\lambda)\|\Lambda^T\Eis(\varphi,\lambda)\|^2\d\lambda,\]
where $\mu_\pi$ is the Langlands parameter of $\pi$ (which is necessarily spherical as a $M(\bbA)$-a representation) and $\tilde{f}$ is the \emph{spherical transform} of $f$.

We choose an $f$ with \emph{bounded support} via the Paley--Wiener theory and the theory of spherical inversion so that spherical transform $\tilde{f}$ localizes at $\mu_\pi+\lambda\approx\mu_{\pi_0}+\lambda_0$, and so that $J_\chi^T(f)$ majorizes the left-hand side of the estimate in \cref{thm:BCF}. Our goal now is to estimate $J_\chi^T(f)$ satisfactorily.

Arthur, in his deep work to develop the spectral side of the trace formula, showed that $J_\chi^T(f)$ is a polynomial in $T$ whose coefficients are given in terms of certain intertwining operators and their derivatives. More precisely, $J_\chi^T(f)$ essentially takes the form of a weighted sum over $P\supset P_0$ of
\[\sum_{\pi\in\Pi_2(M_P)}\sum_{L\supset M_P}\int_{i(\a_L^G)^*}\tilde{f}(\lambda)\tr\left(\calM_L^T(P,\lambda)\mid_{\calA_{\chi,\pi}^2(P)^{K_\infty K}}\right)\d\lambda.\]
The combinatorial formula of Finis--Lapid--M\"uller allows us to understand $\calM_L^T(P,\lambda)$ in terms of a certain linear combination of various rank one intertwining operators and their first derivatives. We can then use the bound on $\calM_L^T(P,\lambda)$ given in \cref{prop:calM bound} that are yielded by properties (TWN+) and (BD), together with Paley--Wiener theory to conclude the proof.

The fact that we only see the first derivative of the intertwiners in the above formula, circumvents the difficulty that arises because of, a priori, high derivatives of the same in the Maass--Selberg formula, as described in \cref{sec:difficulties}. Moreover, the formula above, which is yielded by the theory of \emph{$(G,L)$-families} due to Arthur and Finis--Lapid--M\"uller, bypasses the combinatorial complications that arise in the limit procedure in the Maass--Selberg formula, as described in \cref{sec:difficulties}.

It is worth noting that we do not use the geometric side of the trace formula, rather only use the two different expressions of the \emph{spectral side} given in \cref{prop:pre-trace} (the spectral side of the pre-trace formula) and \cref{prop:spectral side} (the spectral side of the trace formula); also see \cref{prop:spectral-side-full}. We also remark that we crucially use the spectral decomposition for a specific cuspidal datum. This is presumably different than a usual proof of, e.g., Weyl law (see e.g. \cite{finis2021remainder}, \cite{lindenstrauss2007existence}) where one does another sum over \emph{all} cuspidal data in addition to the averages in \cref{thm:BCF}. This extra average will not give us the result of the strength as in, e.g., \cref{thm:gln-bcf}. The primary reason is that the contribution from the cusp forms will dominate. This way we may only achieve a polynomial strength dependence on $\level(K)$ in \cref{thm:gln-bcf} (and \cref{rem:level-gln}) as opposed to a poly-logarithmic strength.

\subsection{On the proof of the applications}\label{subsec:proof of application}
The proofs of the applications do not use \cref{thm:BCF} or \cref{thm:unramified} directly but instead use very similar methods. The reason that we cannot use the theorems directly is explained in \cref{rem:thm3 problems}.

Instead, one can prove the theorems by proving a weighted version of \cref{thm:BCF}, where representations are weighted by how far they are from being tempered. Our method allows such a modification, by using a function $h_{R,q}$ on $\GL_n(\bbA)$ that is essentially the characteristic function of the $R$-radius ball at the archimedean places and of $K(q)$ at the non-archimedean places, where $K(q)$ is the principal congruence subgroup of level $q$ of $\GL_n$ over the finite adeles (see \cref{sec:complexity} below for the definition). This requires increasing $d(T)$ sufficiently slowly so that Arthur's formul{\ae} will still hold. 

In practice, we use the full distribution $J^T$ of Arthur, over the entire spectrum. The idea is to compare two different formul{\ae} of Arthur for $J^T(h_{R,q})$. On the one hand, the spectral side of the trace formula allows us to give a bound on $J^T(h_{R,q})$, given in \cref{prop:assing-blomer}. This is done by using the bound on $\mathcal{M}_L^T(\lambda,P)$ given in \cref{prop:calM bound} to reduce the problem to a sum which is bounded in the work of Assing--Blomer \cite[Theorem~7.1]{assing2022density}.

The other formula for $J^T(h_{R,q})$, which requires that $d(T)$ will be sufficiently large (actually depending logarithmically on $R$), can be seen as a pre-trace formula, and leads to a bound on the usual kernel of $h_{R,q}$, given in \cref{prop: AB reduction1}. The bound on the kernel of $h_{R,q}$ by standard methods (e.g, \cite{sarnak2015lettermiller,golubev2020sarnak}) leads to the proofs of both \cref{thm:counting} and \cref{thm:optimal-lifting}.

We finally remark that similar ideas, of using two formul{\ae} for $J^T(h)$, were recently used by Finis--Lapid to provide a remainder term for Weyl law for a large class of arithmetic non-compact locally symmetric spaces \cite{finis2021remainder}. However, \cite{finis2021remainder} compares the pre-trace formula with the \emph{geometric side} of the trace formula, not the spectral side.

\subsection{Added in proof}
After the completion of this article Assing, Blomer, and Nelson \cite{assing2024local} extended the results of \cite{assing2022density}. In particular, they show that in \cref{thm:counting} and \cref{thm:optimal-lifting} the assumption that $q$ is square-free can be removed.

\subsection{Conventions}
As usual in analytic number theory, we use $\epsilon$ to denote a small positive constant, whose actual value may change from line to line. We use Vinogradov's notations $\gg$ and $\ll$ whose implied constants are understood to depend only on the group $G$ and the number field $F$.

\section*{Acknowledgements}
The authors thank Erez Lapid for several fruitful discussions, his interest in this work, and his encouragement. The authors also thank Valentin Blomer for discussions relating to this work and \cite{assing2022density} and Paul Nelson for several useful comments on an earlier draft of this paper. The first author thanks Shreyasi Datta for various enlightening conversations and her immense support during this project. The first author also thanks MPIM Bonn where most of the work was completed during his stay there. The second author is supported by the ERC under the European Union's Horizon 2020 research and innovation programme (grant agreement No. 803711). Finally, the authors generously thank the anonymous referees for carefully reading the paper and suggesting numerous improvements.

\section{General Preliminaries}

Let $F$ denote a number field with integer ring $\mathfrak{o}$. The letter $v$ will usually (unless mentioned otherwise) denote an arbitrary place of $F$ and $F_v$ will denote the $v$-adic completion of $F$. Let $G$ denote a reductive group defined over $F$. Let $\bbA$ denote the adele ring of $F$. Also let $\bbA_f$ denote the finite adeles and $F_\infty := F\otimes\R$.

Below we briefly discuss the main ingredients and quantities in our proof. We primarily follow the notations of \cite[\S 2.1]{finis2011spectral-abs} and \cite[\S 4]{finis2015limit}.

\subsection{Brief structure theory of reductive groups}

Let $T_0$ denote a fixed maximal $F$-split torus of $G$. Let $M_0$ denote the centralizer of $T_0$ in $G$, which is a minimal Levi subgroup. Let $\calL$ denote the finite set of Levi subgroups defined over $F$ containing $M_0$. Given $M\in \calL$, let $\calL(M) \subset \calL$ denote the (finite) set of Levi subgroups containing $M$ and $\calP(M)$ denote the (finite) set of parabolic subgroups defined over $F$ with Levi subgroup $M$. We also fix a minimal parabolic subgroup $P_0 \supset M_0$. 

Let $T_M$ denote the split part of the identity component of the center of $M$. We denote $W(M) := N_{G}(M)/M$ where $N_G(M)$ denotes the normalizer of $M$ in $G$. We denote the Weyl group of $(G,T_0)$ by $W_0 := W(M_0)$. We can identify $W(M)$ as a subgroup of $W_0$. 

Let $\a_M^*$ denote the $\R$-vector space spanned by the lattice $X^*(M)$ of $F$-rational characters of $M$, and let $\a_{M,\C}^* := \a_M^*\otimes_\R \C$. For $M_1 \subset M_2$ it holds that $\a_{M_2}^* \subset \a_{M_1}^*$ with a well-defined complement $\a_{M_1}^* = \a_{M_2}^*\oplus (\a_{M_1}^{M_2})^*$. Let $\a_M$ denote the dual space of $\a_M^*$ and we fix a pairing $\langle,\rangle$ between $\a_M$ and $\a_M^*$. For $P\in \calP(M)$ we write $\a_P =\a_M$.

We choose a $W_0$-invariant inner product on $\a_0:= \a_{M_0}$, which also defines measures on $\a_M$ and $\a_M^*$ for all $M\in \calL$. We normalize all measures as in \cite[\S 7]{arthur2005intro}, similar to \cite{finis2011spectral-abs}.

For $P\in \calP(M)$, we write $M_P=M$. Let $N_P$ denote the unipotent radical of $P$. Let $\Sigma_P \subset (\a_{P}^{G})^*$ denote the set of roots of $T_M$ on the Lie algebra of $N_P$. Let $\Delta_P$ denote the subset of simple roots of $\Sigma_P$. Let $\Sigma_M := \cup_{P\in \calP(M)}\Sigma_P$. Given two parabolics $P,Q\in \calP(M)$ we say that $P$ and $Q$ are \emph{adjacent along $\alpha \in \Sigma_M$} and write $P|^{\alpha} Q$ if $\Sigma_P = \Sigma_Q \setminus \{-\alpha\} \cup\{\alpha\}$.

Let $A_0$ denote the identity component of $T_0(\R)$. We let $A_M := A_0 \cap T_M(\R)$. Let 
$H_M\colon  M(\bbA)\to \a_M$
denote the natural homomorphism defined by 
\[
e^{\langle \chi, H_M(m) \rangle} := |\chi(m)|_\bbA = \prod_v |\chi(m_v)|_v
\]
for any $\chi \in X^*(M)$. Here $|\cdot|_\bbA$ denotes the homomorphism from $\bbA^\times\to\R^+$ given by $|\cdot|_\bbA:=\prod_v|\cdot|_v$ where $|\cdot|_v$ denotes the usual absolute value on $F_v$. Let $M(\bbA)^1 \subset M(\bbA)$ denote the kernel of $H_M$. It holds that $M(\bbA) = A_M \times M(\bbA)^1$.

We fix a maximal compact subgroup $K_\infty$ of $G(F_\infty)$ and a maximal open-compact subgroup $K_f$ of $G(\bbA_f)$. Finally, we fix the maximal compact subgroup $K_\bbA := K_\infty K_f$ of $G(\bbA)$, and assume that it is admissible relative to $M_0$, so that for every $P\in \calP(M_0)$ an Iwasawa decomposition $G(\bbA) = P(\bbA)K_\bbA$ holds.

Given $M\in \calM$ and $P\in \calP(M)$, let $H_P\colon G(\bbA) \to \a_P$ denote the extension of $H_M$ to a left $N_P(\bbA)$ and right $K_\bbA$-invariant map.

\section{Global Preliminaries}

\subsection{Automorphic forms}\label{sec:automorphic-forms}

Given $M\in \calL$, let 
\[
L^2_{\disc}(M(F) \backslash M(\bbA)^1) \cong L^2_{\disc}(M(F)A_M \backslash M(\bbA))
\]
denote the discrete part of the spectrum of $L^2(M(F)A_M\backslash M(\bbA))$, i.e. the Hilbert sum of all irreducible subrepresentations of $M(\bbA)$ on the space.

Let $\Pi_2(M)$ denote\footnote{In \cite{finis2011spectral-abs} the authors denote it by $\Pi_{\disc}(M(\bbA))$.} the countable set of equivalence classes of irreducible unitary representation of $M(\bbA)$ appearing in $L^2_{\disc}(M(F) \backslash M(\bbA)^1)$. Each element $\pi \in \Pi_2(M)$ is an abstract unitary representation of $M(\bbA)$. By abuse of notations, we consider each $\pi \in \Pi_2(M)$ also as a representation of $M(\bbA)^1$. 
By Flath's tensor product theorem (see \cite{flath1979decomposition}), each $\pi \in \Pi_2(M)$ can be considered as a tensor product of local representations $\pi = \otimes_v \pi_v$, where $\pi_v$ is a representation of $M(F_v)$. We denote the finite part of the representation $\otimes_{v<\infty}\pi_v$ by $\pi_f$, and the infinite part of the representation $\otimes_{v\mid \infty}\pi_v$ by $\pi_\infty$.

Let $P\in\calP(M)$ and $\calA^2(P)$ denote the set of automorphic forms $\varphi$ on $N_P(\bbA)M(F)\backslash G(\bbA)$ that are square-integrable, that is, $\varphi$ satisfies 
\[
\delta_P^{-1/2}\varphi(\cdot k) \in L^2_\disc(A_MM(F) \backslash M(\bbA)),\quad \forall k\in K_\bbA,
\]
where $\delta_P$ is the modular character attached to $P$ (in particular, $\varphi(amk)=\delta_P^{1/2}(a)\varphi(mk)$ for $a\in A_M$, $m\in M(\bbA)^1$, and $k\in K_{\bbA}$). Then
$\calA^2(P)$ is a pre-Hilbert space, with an inner product 
\[
\langle\varphi_1,\varphi_2\rangle_{\calA^2(P)} := \intop_{A_M M(F) N_P(\bbA) \backslash G(\bbA)} \varphi_1(g) \overline{\varphi_2(g)} \d g.
\]
Sometimes, we shorthand $\|\cdot\|_{\calA^2(G)}$ as $\|\cdot\|_2$.
Let $\overline{\calA^2(P)}$ denote the Hilbert space completion of $\calA^2(P)$. 
For $\pi \in \Pi_2(M)$ we will also denote $\calA^2_\pi(P)\subset \calA^2(P)$ to be the subspace of $\varphi$ such that $\delta_P^{-1/2}\varphi(\cdot k)$ lies in the $\pi$-isotypic subspace of $L^2_{\disc,\pi}(A_M M(F) \backslash M(\bbA))\subset L^2_\disc(A_M M(F) \backslash M(\bbA))$ for every $k\in K_\bbA$.

Given $\lambda \in i(\a_P^G)^*$, let $\rho(P,\lambda,\cdot)$ denote the unitary representation of $G(\bbA)$ on $\overline{\calA^2(P)}$, given by 
\[\rho(P,\lambda,y)(\varphi)(x) := \varphi(xy)e^{\langle \lambda, H_P(xy)-H_P(x) \rangle},\quad y\in G(\bbA).\]
The representation $\rho(P,\lambda,\cdot)$ is isomorphic to 
\[
\Ind_{P(\bbA)}^{G(\bbA)}\left(L^2_\disc(A_M M(F) \backslash M(\bbA))\otimes e^{\langle \lambda, H_M(\cdot) \rangle}\right);
\]
see \cite[\S 4.2]{finis2015limit}.
We remark that the action of every compact subgroup of $G(\bbA)$ on $\calA^2(P)$ does not depend on $\lambda$.

Recall that a \emph{cuspidal datum} $\chi$ is a $G(F)$-conjugacy class of pairs $(L,\sigma)$, where $L$ is a Levi subgroup of $G$ defined over $F$ and $\sigma$ is a cuspidal representation of $L(\bbA)$; see \cite[\S 12]{arthur2005intro} for a detailed discussion.
We have a coarse spectral decomposition
\[L^2_\disc(A_M M(F) \backslash M(\bbA)) = \widehat{\bigoplus}_{\chi}L^2_{\disc,\chi}(A_M M(F) \backslash M(\bbA));\]
see \cite[eq.(12.5)]{arthur2005intro}. Here and elsewhere $\widehat{\oplus}$ denotes a Hilbert space direct sum. We have a similar coarse spectral decomposition of $\calA^2(P)$; see \cite[p.66-67]{arthur2005intro}. Let $\calA^2_\chi(P)$ denote the $\chi$-part of this spectral decomposition. Finally, we also write 
\begin{equation*}
L^2_{\disc,\chi,\pi}(A_M M(F) \backslash M(\bbA)) := L^2_{\disc,\chi}(A_M M(F) \backslash M(\bbA)) \cap L^2_{\disc,\pi}(A_M M(F) \backslash M(\bbA))
\end{equation*}
and $\calA^2_{\chi,\pi}(P) := \calA^2_{\chi}(P) \cap \calA^2_{\pi}(P)$.

\subsection{Some discussion of the coarse spectral expansion}\label{sec:coarse-spectral-expansion}

Since this topic is less standard, for the convenience of the reader, we add here a discussion of the coarse spectral decomposition. 

We start with the group $\GL_n$, in which case the coarse spectral decomposition is well understood. Thanks to the multiplicity one theorem, each cuspidal representation $\sigma$ of a Levi subgroup $L$ appears with multiplicity one, i.e., 
\[
\dim \Hom_{L(\bbA)}(\sigma, L^2(A_L L(F) \backslash L(\bbA))) \le 1. 
\]

Next, let us describe some of the results of \cite{moeglin1989residual}. First, we recall that given a cuspidal representation $\sigma$ of $\GL_a(\bbA)$ and $b>0$, one may construct a Speh representation $\operatorname{Speh}(\sigma,b)$ of $\GL_{ab}(\bbA)$ which is the unique irreducible quotient of the induction of $\sigma^{\otimes b}$ from the Levi subgroup $\GL_a(\bbA)^b$ and twisted by the character $\GL_a(\bbA)^b\ni (g_1,g_2,\dots,g_b)\mapsto|\det(g_1)|^{(b-1)/2}|\det(g_2)|^{(b-3)/2}\dots|\det(g_b)|^{(1-b)/2}$. It holds that $\operatorname{Speh}(\sigma,b) \in \Pi_2(\GL_{ab})$, and each representation $\pi \in \Pi_2(\GL_n)$ is isomorphic to a unique Speh representation $\operatorname{Speh}(\sigma,b)$ for some $b\mid n$ and cuspidal representation $\sigma$ of $\GL_{n/b}(\bbA)$. Note that the $b=1$ case corresponds to the cuspidal representations in $\Pi_2(\GL_n)$.

Now, if $M$ is a Levi subgroup of $\GL_n$ of a parabolic $P$, for each $\pi \in \Pi_2(M)$ there is a unique cuspidal datum $\chi$ such that $\calA^2_{\chi,\pi}(P)\ne \{0\}$, and the multiplicity of $\pi$ in the space $L^2_\disc(A_MM(F) \backslash M(\bbA))$ is equal to $1$. More precisely, if $M \cong \GL_{n_1}\times\dots\times\GL_{n_k}$ is the standard Levi subgroup, and $\pi = \pi_1 \otimes \dots \otimes \pi_k$, then one can write (uniquely) $n_i = a_i b_i$, find a cuspidal representation $\sigma_i \in \Pi_2(\GL_{a_i})$ and write $\pi_i = \operatorname{Speh}(\sigma_i,b_i)$. In this case, let $L:=\GL_{a_1}^{b_1}\times\dots\times\GL_{a_k}^{b_k}$ and $\sigma := \sigma_1^{\otimes b_1}\otimes\dots\otimes\sigma_k^{\otimes b_k} \in \Pi_2(L)$ be cuspidal. If $\chi$ is the cuspidal datum corresponding to $(L,\sigma)$ then $\calA^2_{\chi,\pi}(P)\ne \{0\}$.

The procedure can be reversed, but one should be a bit careful about the fact that cuspidal datum is only defined up to equivalence. In general, given a Levi $L$ and $\sigma \in \Pi_2(L)$ cuspidal let $\chi:=(L,\sigma)$ be the cuspidal datum and $P$ be a standard parabolic containing $L$ as a Levi subgroup. Then there are at most $n_P:= |\sum_{P'\sim P} W(\a_{P},\a_{P'})|\le |W_0|=n!$ other pairs $(L',\sigma')$ of Levi subgroups $L'\supset M_0$ and cuspidal representations $\sigma'\in \Pi_2(L')$ in $\chi$; see \cite[\S12]{arthur2005intro}.
Given such a pair $(L,\sigma)$
and a Levi subgroup $L\subset M$, we get at most one residual representation in $L^2_{\disc,\chi}(A_M M(F) \backslash M(\bbA))$. All in all, there are at most $n!$ pairs of standard parabolics $P$ and $\pi\in \Pi_2(M_P)$ such that $\calA^2_{\chi,\pi}(P)\ne \{0\}$.



For groups that are different than $\GL_n$ the situation becomes much more involved. An interesting case is the group $G=G_2$. By \cite[Appendix III]{moeglin1995spectral}, for $\chi= (M_0,\triv)$ (the torus and the trivial representation of it), there are infinitely many different $\pi\in \Pi_2(G)$ such that
\[
L^2_{\disc,\chi,\pi}(A_G G(F) \backslash G(\bbA))\ne \{0\}.
\]
See also \cite{kim1996residual, zampera1997residal} for complete results about the residual spectrum in this case. 
Moreover, by \cite{gan2002cubic}, for some $\pi \in \Pi_2(G)$ there are two different cuspidal data $\chi_1,\chi_2$ such that  
\[
L^2_{\disc,\chi_i,\pi}(A_G G(F) \backslash G(\bbA))\ne \{0\},
\]
with $\chi_1$ associated with the subgroup $G$ and $\chi_2$ associated with a proper parabolic subgroup. Finally, the multiplicity of $\pi$ in $L^2_{\disc,\chi_1,\pi}(A_G G(F) \backslash G(\bbA))$ can be arbitrarily large.

\subsection{Eisenstein series}\label{sec:eisenstein-series}

Given $\varphi \in \calA^2(P)$ and $\lambda \in (\a_{P,\C}^G)^*$, we define a corresponding Eisenstein series, as
\[
\Eis(\varphi,\lambda)(g) := \sum_{\gamma \in P(F) \backslash G(F)} \varphi(\gamma g) e^{\langle \lambda, H_P(\gamma g)\rangle }.
\]
The above converges absolutely if $\lambda$ is sufficiently dominant\footnote{We call $\lambda$ to be \emph{dominant} (resp.\ \emph{sufficiently dominant}) if $\langle \lambda,\alpha\rangle\ge 0$ (resp.\ $\langle \lambda,\alpha\rangle$ are sufficiently positive) for all positive root $\alpha$}. One can then, via Langlands' deep work (see e.g.\ \cite{langlands1966eisenstein}), meromorphically continue $\Eis(\varphi,\lambda)$ to all $\lambda\in (\a_{P,\C}^G)^*$.

It holds that the Eisenstein series intertwines the representation $\rho(P,\lambda,\cdot)$. Namely, it holds that \[
\Eis(\varphi,\lambda)(gy) = \Eis(\rho(P,\lambda,y)\varphi,\lambda)(g),
\]
for $y\in G(\bbA)$.

\subsection{Truncation}\label{sec:truncation}

Recall the minimal parabolic subgroup $P_0 \in \calP(M_0)$. Given $T\in \a_{M_0}^G$, let 
\[
d(T):= \min_{\alpha \in \Delta_{P_0}}\{\alpha(T)\}.
\]
Given $T$ with $d(T)$ sufficiently large, Arthur (following Langlands; see \cite[\S 13]{arthur2005intro}) defined certain \emph{truncation operator} $\Lambda^T$ which acts on locally integrable functions on $G(F) \backslash G(\bbA)^1$.
We do not need the exact definition of $\Lambda^T$, but just the following properties of it.

Let $f$ be a locally integrable function on $G(F)\backslash G(\bbA)^1$.
\begin{enumerate}
    \item $\Lambda^Tf$ is rapidly decaying. In particular, if $\lambda \in (\a_{P,\C}^G)^*$ is away from any pole of $\Eis(\varphi,\lambda)$ then $\Lambda^T\Eis(\varphi,\lambda)$ is square-integrable.
    \item Given any compact $\Omega \subset G(F)\backslash G(\bbA)^1$ there is a $C(\Omega)>0$ such that for every $T\in\a_{M_0}^G$ with $d(T)>C(\Omega)$ it holds that $\Lambda^Tf\mid_\Omega=f\mid_\Omega$; see \cite[Lemma 6.2]{lapid2006fine}.
   \item When $\alpha(T)\le \alpha(T')$ for every $\alpha \in \Delta_{P_0}$ it holds that $\Lambda^{T'}\Lambda^{T}=\Lambda^T$ and $\|\Lambda^{T'}f\|_2\ge\|\Lambda^T f\|_2$; see \cite[\S3.3]{finis2021remainder}.
\end{enumerate}

\subsection{Intertwining operators}

We follow \cite[\S4.2]{finis2015limit} to define intertwining operators and to record their relevant properties.

For $P,Q \in \calP(M)$ and $\lambda\in\a^*_{M,\C}$ we define the standard \emph{intertwining operator}
\[
M_{Q|P}(\lambda)\colon \calA^2(P) \to \calA^2(Q)
\]
by
\begin{equation}\label{def-restricted-intertwiner}
(M_{Q|P}(\lambda)\varphi)(x) :=
\intop_{N_Q(\bbA)\cap N_P(\bbA) \backslash N_Q(\bbA)}\varphi(nx) e^{\langle \lambda, H_P(nx) - H_Q(x)\rangle} \d n.
\end{equation}
The above integral converges absolutely for sufficiently dominant $\lambda$. It can then be meromorphically continued for all $\lambda\in\a^*_{M,\C}$. It can be checked that $M_{Q\mid P}(\lambda)^{-1}=M_{P\mid Q}(\lambda)$ and for $\lambda\in i(\a_M^G)^*$ the operator $M_{Q\mid P}(\lambda)$ is unitary.

For each $\pi\in\Pi_2(M)$ we denote the 
restriction of $M_{Q\mid P}(\lambda)$ on $\calA^2_\pi(P)$ by $M_{Q\mid P}(\pi,\lambda):\calA^2_\pi(P)\to\calA^2_\pi(Q)$.
On the other hand, we can define an intertwiner from
\[M^{\mathrm{I}}_{Q\mid P}(\lambda,\pi):\ind_{P(\bbA)}^{G(\bbA)}\pi\to\ind_{Q(\bbA)}^{G(\bbA)}\pi\]
that can be densely defined on $K$-finite and $Z\left(\mathrm{Lie}(G(F_\infty))\otimes_\R\C\right)$-finite subspace of $\ind_{P(\bbA)}^{G(\bbA)}\pi$ as in \eqref{def-restricted-intertwiner} (first, for sufficiently dominant $\lambda$ then by meromorphic continuation).
Then we have a unique canonical isomorphism of $G(\bbA_f)\times\left(\mathrm{Lie}(G(F_\infty))\otimes_\R\C,K_\infty\right)$-modules (see \cite[\S4.2]{finis2015limit})
\[
j_P\,\colon\, \Hom_{M(\bbA)}\left(\pi,L^2(A_MM(F) \backslash M(\bbA))\right)\otimes \ind_{P(\bbA)}^{G(\bbA)}\pi \to \calA_\pi^2(P)
\]
(similarly, $j_Q$) that can be characterized by
\[j_Q\circ\left[\mathrm{Id}\otimes M^{\mathrm{I}}_{Q\mid P}(\lambda,\pi)\right]=M_{Q\mid P}(\lambda,\pi)\circ j_P;\]
see \cite[eq.(2.3)]{muller2002spectral} and discussion around it.

Endowing $\Hom_{M(\bbA)}\left(\pi,L^2(A_MM(F) \backslash M(\bbA))\right)$ with the inner product $\langle f_1,f_2\rangle = f_2^*f_1$, where $*$ denotes adjoint (notice that by Schur's lemma, $f_2^* f_1$ is a scalar), we note that $j_P$ is an isometry. 

Now in the rest of this subsection let $P|^\alpha Q$ for some $\alpha\in\Sigma_M$. In this case, $M_{Q\mid P}(\lambda)$ depends only upon $\langle \lambda,\alpha^\vee\rangle$. Let $\varpi \in \a_M^*$ be such that $\langle \varpi,\alpha^\vee\rangle=1$.
We have
\begin{equation}\label{factorization-intertwiner}
    M_{Q|P}(s\varpi)\mid_{\calA^2_\pi(P)} \circ j_P = n_\alpha(\pi,s) \cdot j_Q \circ \left(\mathrm{Id}\otimes R_{Q|P}(\pi,s)\right),
\end{equation}
for any $s\in\C$ avoiding poles of the operators on both sides of above.
Here $n_\alpha(\pi,s)$, a \emph{global normalizing factor}, is a meromorphic function in $s$. Moreover, if $s\in i\R$ then $|n_\alpha(\pi,s)|=1$. The operator 
\[R_{Q|P}(\pi,s):\ind_{P(\bbA)}^{G(\bbA)}\pi\to \ind_{Q(\bbA)}^{G(\bbA)}\pi\]
is a factorizable \emph{normalized intertwining operator}, so that 
\[R_{Q|P}(\pi,s)=\prod_v R_{Q|P}(\pi_v,s),\quad\pi = \otimes_v \pi_v.\]
Here $R_{Q\mid P}(\pi_v,s)\One_{K_v}=\One_{K_v}$ for almost all $v$. Moreover, if $s\in i\R$ then $R_{Q\mid P}(\pi,s)$ is unitary. See \cite[Theorem 21.4]{arthur2005intro} for details.

\subsection{Langlands inner product formula}

For $Q\in \calP(M)$ we let 
\[
\theta_Q(\Lambda) := v_{\Delta_Q}^{-1}\prod_{\beta \in \Delta_Q}\langle \Lambda,\beta^\vee \rangle,\quad \Lambda\in \a_{M_0,\C}^*
\]
where $v_{\Delta_Q}$ is the co-volume of the lattice in $\a_Q^G$ spanned by $\{\beta^\vee \mid \beta \in \Delta_Q\}$.

For $T\in\a_{M_0}^G$ we also define $Y_Q(T)$ as the projection to $\a_M^G$ of $t^{-1}(T-T_0)+T_0$, where $t \in W_0$ is an element such that $P_0 \subset tQ $, and $T_0 \in a_M^G$ is the element defined by Arthur in \cite[Lemma 1.1]{arthur1981trace} (also see the discussion around \cite[eq.(9.4)]{arthur2005intro}).

Let $\chi = (M,\pi)$ be a cuspidal datum and $\varphi,\varphi' \in \calA^2_{\chi,\pi}(P)$, namely, $\varphi$ and $\varphi'$ are cusp forms. In this case, Langlands, generalizing the classical Maass--Selberg relations obtained a formula for the inner product of two truncated Eisenstein series attached to $\varphi$ and $\varphi'$; see \cite[Proposition 15.3]{arthur2005intro}.
The formula involves sums over various Weyl elements and parabolic subgroups of quantities involving (twisted) intertwining operators, $\theta_Q$, and $Y_Q$. There is also an approximate version of the formula, due to Arthur in \cite{arthur1982inner}, when $\varphi$ or $\varphi'$ are not cuspidal. However, we do not need any of these formul{\ae} here, but only a special case that we describe below.

For any two parabolic subgroups $P,Q\supset P_0$ let $W(\a_P,\a_Q)$ be the set of distinct linear isomorphisms from $\a_P\subset\a_{P_0}$ to $\a_Q\subset \a_{P_0}$, induced by elements of $W_0$. For any $Q\in\calP(M)$ and $P_0\subset P$ there is a natural action of $W(\a_P,\a_Q)$ on $\Pi_2(M)$. Following \cite[\S15]{arthur2005intro}, we call $\chi$ \emph{unramified} if for every $(M,\pi)\in\chi$ the stabilizer of $\pi$ in $W(\a_P,\a_P)$ is trivial. We warn again that this notion should not be confused with the unramifiedness of representations of reductive groups over local fields.

In the unramified case, we record the Langlands inner product formula.
\begin{lemma}\label{lem:Maass-Selberg unram}
Let $\chi:=(M,\pi)$ be any unramified cuspidal datum and $\varphi,\varphi'\in\calA_{\chi,\pi}^2(P)$. Then we have
\begin{equation*}
 \langle \Lambda^T \Eis(\varphi,\lambda), 
\Lambda^T \Eis(\varphi',\lambda) \rangle_{\calA^2(G)} = 
\lim_{\Lambda \to 0} \sum_{Q\in \calP(M)} e^{\langle \Lambda, Y_Q(T) \rangle}\frac{\langle M_{Q|P}(\lambda+\Lambda)\varphi,M_{Q|P}(\lambda)\varphi'\rangle_{\calA^2(Q)}}{\theta_Q(\Lambda)}.   
\end{equation*}
\end{lemma}

\begin{proof}
We refer to \cite[pp.185]{finis2011spectral-abs} and note that $M_{Q\mid P}(1,\lambda)$ in this reference is the same as $M_{Q\mid P}(\lambda)$.
\end{proof}

\section{Local Preliminaries}
\subsection{Parameters associated with a representation}\label{sec:complexity}

Let $M\in\calL$. The goal of this subsection is to associate various parameters measuring the complexity of a certain representation $\pi \in \Pi_2(M)$.

\subsubsection{Non-archimedean complexity}

First, following \cite[\S 5.1]{finis2015limit}, for any open-compact subgroup $K\subset K_f$ we recall the notion $\level(K)$. Fix $\iota:G\to \GL(V)$ a faithful $F$-rational representation and an $\frako$-lattice $\Lambda\subset V$ so that the stabilizer of $\hat{\frako}\otimes \Lambda\subset \bbA_f\otimes V$ in $G(\bbA_f)$ is $K_f$. Then for any ideal $\frakq\subset\frako$ we define the \emph{principal congruence subgroup of level $\frakq$} by
\[K({\frakq}):=\{k \in K_f\mid \iota(k)v \equiv v \mod(\frakq(\hat{\frako}\otimes\Lambda)), v\in \hat{\frako}\otimes\Lambda\}.\]
Now we define $\level(K)$ to be the index $[\frako:\frakq_K]$ where $\frakq_K\subset \frako$ be the largest ideal $\frakq$ such that $K\supset K({\frakq})$. On the other hand, given a representation $\pi_f$ of $G(\bbA_f)$, we define $\level(\pi_f)$ as the the minimum over $\level(K)$, such that $\pi_f$ has a non-trivial $K$-invariant vector.\footnote{We remark that when $G=\GL_n$, the above definition of the level of a representation of is different than the \emph{usual} definition of the level which is defined with respect to the \emph{Hecke congruence subgroup} $K_0(\frakq)$; see \cite{jacquet1981conducteur}, also \cite[Lemma 5.6]{finis2015limit}.}

Given a representation $\pi_f$ of $M(\bbA_f)$ and $\alpha\in\Sigma_M$, we recall the quantity $\level_M(\pi_f,\hat{M}_\alpha^+)$, as defined in \cite[\S5.1-5.2]{finis2015limit}, which measures the complexity of $\pi_f$. In this paper, we do not need to know $\level_M(\pi_f,\hat{M}_\alpha^+)$ explicitly, but only an upper bound that we describe below.

\begin{lemma}\label{bound-of-level}
Let $K\subset K_f$ be an open-compact subgroup. Assume that $\pi\in \Pi_2(M)$ is such that $\calA^2_\pi(P)^{K} \neq \{0\}$ for some $P\in\calP(M)$, i.e., the parabolic induction of $\pi_f$ to $G(\bbA_f)$ has a $K$-invariant vector. Then we have
\[
\level_M(\pi_f,\hat{M}_\alpha^+)\le \level(K).
\]
\end{lemma}

\begin{proof}
It follows from \cite[Lemma 5.12]{finis2015limit} that 
\[\level_M(\pi_f,\hat{M}_\alpha^+)\le \level(K;G_M^+),\]
where $G_M^+\subset G$ is a certain subgroup and $\level(K;G_M^+)$ denotes the relative level of $K$ respect to $G_M^+$ as defined in \cite[\S 5.1]{finis2015limit}. The lemma follows from the fact that $\level(K;G_M^+)\le \level(K)$; see \cite[\S 5.1]{finis2015limit}.
\end{proof}

\subsubsection{Archimedean complexity}

Let $\pi_\infty$ be a smooth representation of $M(F_\infty)$. 
Let $\alpha\in\Sigma_M$. We recall the quantity $\Lambda_M(\pi_\infty;\hat{M}_\alpha)$, as defined in \cite[\S 5.1-5.2]{finis2015limit}, which measures the complexity of $\pi_\infty$ as a representation of $M(F_\infty)$. In this paper, we do not need to know $\Lambda_M(\pi_\infty;\hat{M}_\alpha)$ explicitly, but only an upper bound that we describe below. 

Let $\nu(\pi_\infty)$ be the eigenvalue for $\pi_\infty$ of the Casimir operator of $M(F_\infty)$. Similarly, given $\tau\in\widehat{K_\infty}$, let $\nu(\tau)$ be the Casimir eigenvalue of $\tau$.

\begin{lemma}\label{bound-of-infinitesimal-character}
Let $\tau\in\widehat{K_\infty}$. Assume that $\pi\in\Pi_2(M)$ is such that $\calA_\pi^2(P)^{\tau}\neq \{0\}$ for some $P\in\calP(M)$, i.e., the parabolic induction of $\pi_\infty$ to $G(F_\infty)$ has a nonzero vector whose $K_\infty$-type is $\tau$. Then we have
\[\Lambda_M(\pi_\infty;\hat{M}_\alpha)\ll 1+\nu(\pi_\infty)^2+\nu(\tau)^2.\]
\end{lemma}

\begin{proof}
It follows from \cite[eq.(10)-(11)]{finis2015limit} that
\[\Lambda_M(\pi_\infty;\hat{M}_\alpha) \ll \nu(\pi_\infty)^2+\min_{\tau^\prime}\nu(\tau^\prime)^{2},\]
where $\tau'\in \widehat{K_\infty}$ runs over the lowest $K_\infty$-types of $\Ind_{P(F_\infty)}^{G(F_\infty)}\pi_\infty$. We conclude by noting that the lowest $K_\infty$-types have the smallest Casimir eigenvalues (equivalently, smallest infinitesimal characters).
\end{proof}

\subsection{The spherical function}\label{subsec:spherical function}
In the rest of this section, we only work over the archimedean places, i.e. $F_\infty$. See \cite[\S 3]{Duistermaat1979spectra} for references to the material below.

Let $G_\infty := A_G \backslash G(F_\infty)$ and $K_\infty$ be a maximal compact subgroup in $G_\infty$.
Choose a maximal $\R$-split torus $T_\infty \subset G_\infty$, which we may assume to contain $T_0(F_\infty)$ and contained in $M_0(F_\infty)$. Let $W_\infty:= N_{G_\infty}(T_\infty)/T_\infty$ be the Weyl group of $G_\infty$ relative to $T_\infty$.
Let $T_\infty^0$ be the connected component of the identity of $T_\infty$ and $\a_\infty$ be the Lie algebra of $T_\infty^0$. 

We remark that for any Levi $M\supset M_0$ the embedding $T_\infty \to M(F_\infty)$ defines a projection $\a_\infty \twoheadrightarrow \a_{M}^G$, and an embedding $(\a_{M}^G)^* \hookrightarrow \a_\infty^*$; see \cite[pp.118]{arthur2005intro}.

We let $M_{0,\infty}$ be the Levi subgroup which is the normalizer of $T_\infty$ in $G_\infty$. Let $B_\infty$ be a Borel subgroup containing $M_{0,\infty}$ as a Levi subgroup and also containing $P_0(F_\infty)$. Let $N_\infty$ be the unipotent radical of $B_\infty$. We have the Iwasawa decomposition $G_\infty = B_\infty K_\infty = N_\infty T_\infty^0 K_\infty$.

Using the exponential map from $\a_\infty$ to $T_\infty^0$, each element $\mu \in \a_{\infty,\C}^*:=(\a_\infty\otimes_\R\C)^*$ defines a character $\chi_\mu$ of $T_\infty$, which is trivial on $T_\infty/T_\infty^0$. It extends, via $N_\infty$-invariance, to a character of $B_\infty$.
Let $\delta_\infty\colon B_\infty \to \R_{>0}$ be the modular character of $B_\infty$, which is attached to the character $\chi_{2\rho_\infty}$, where $2\rho_\infty$ is the sum of all the roots of the action of $T_\infty$ on the Lie algebra of $N_\infty$, weighted by the dimensions of the corresponding root spaces.

Given $g \in G_\infty$, we let $a(g)$ be its $T_\infty^0$ part according to the Iwasawa decomposition $G_\infty = N_\infty T_\infty^0 K_\infty $. 
We define the \emph{spherical function} $\eta_\mu\colon G_\infty \to \C$ corresponding to $\mu \in \a_{\infty,\C}^*$ by 
\[
\eta_\mu(g) := \intop_{K_\infty}\chi_{\mu+\rho_\infty}(a(kg))\d k.
\]
Thus $\eta_\mu$ is bi $K_\infty$-invariant. Here $\d k$ denotes the probability Haar measure on $K_\infty$.

\subsection{Spherical representation}

We call an irreducible admissible representation $(\pi,V)$ of $G_\infty$ \emph{spherical} if $\pi$ has a non-zero $K_\infty$-invariant vector. We can construct all irreducible admissible spherical representations of $G_\infty$ from the unitarily induced principal series representations. Let $\mu\in\a_{\infty,\C}^*$ and $\ind_{B_\infty}^{G_\infty}\chi_\mu$ denotes the normalized parabolic induction of $\chi_\mu$ from $B_\infty$ to $G_\infty$. It is an admissible representation and has a unique irreducible spherical subquotient. Conversely, for any irreducible  spherical representation $(\pi,V)$ we can find a (unique up to $W_\infty$-action) $\mu_\pi\in\a_{\infty,\C}^*$ such that $\pi$ appears as the unique spherical subquotient of $\ind_{B_\infty}^{G_\infty}\chi_{\mu_\pi}$. In this case, we call $\mu_\pi$ the \emph{Langlands parameter} of $\pi$.  
When $\pi$ is also unitary it holds that $\|\Re(\mu_\pi)\| \le \|\rho_\infty\|$ and $1+\|\mu_\pi\|^2\asymp 1+ \nu(\pi)$ (more precisely, $\nu(\pi) = \|\rho_\infty\|^2- \|\Re(\mu_\pi)\|^2 + \|\Im(\mu_\pi)\|^2$); see \cite[eq.(3.17)]{Duistermaat1979spectra}.

If $\pi$ is spherical then the $K_\infty$-invariant vector of $\pi$ is unique up to multiplication by a scalar. Let $\One_{K_\infty}$ be the characteristic function on $K_\infty$. Then $\pi(\One_{K_\infty})$ is well-defined and acts as a projection onto the $K_\infty$-invariant subspace of $V$. Let $0\neq v\in\pi$ be a $K_\infty$-invariant vector. It follows from the definition of the spherical function that $\pi(\One_{K_\infty})\pi(g)v=\eta_{\mu_\pi}(g)v$.

Let $M_\infty$ be any Levi subgroup of $G_\infty$, having $T_\infty$ as a maximal split torus. We can associate to each irreducible admissible spherical (containing a non-zero $K_\infty \cap M_\infty$-invariant vector) representation $\pi$ of $M_\infty$ a Langlands parameter $\mu_\pi \in \a_{\infty,\C}^*$, defined up to $W_{M,\infty}=N_{M_\infty}(T_\infty)/T_\infty\subset W_\infty$. Moreover, if $Q_\infty$ is a parabolic subgroup of $G_\infty$ with $M_\infty$ as its Levi subgroup, $\Ind_{Q_\infty}^{G_\infty}\pi$ has a unique spherical subrepresentation. On the other hand, if $\Ind_{Q_\infty}^{G_\infty}\pi$ has a $K_\infty$-invariant vector then $\pi$ necessarily has a $K_\infty \cap M_\infty$-invariant vector. In any case, by the transitivity of induction, $\mu_\pi$ is also the Langlands parameter of the spherical subquotient of $\Ind_{Q_\infty}^{G_\infty}\pi$ (however, the new parameter is defined up to $W_\infty$). 

\subsection{The spherical transform}\label{sec:spherical-transform}

We call the convolution algebra on $C_c^\infty(K_\infty\backslash G_\infty / K_\infty)$ the \emph{spherical Hecke algebra} of $G_\infty$. For $h \in C_c^\infty(K_\infty \backslash G_\infty / K_\infty)$, we let $\tilde{h}\colon \a_{\infty,\C}^* \to \C$ to be the \emph{spherical transform} of $h$, defined by
\[
\tilde{h}(\mu) := \intop_{G_\infty } h(g) \eta_\mu(g) \d g.
\]
If $\pi$ is spherical and $0\neq v\in\pi$ is a $K_\infty$-invariant vector then it holds that
\[\pi(h)v = 
\intop_{G_\infty} h(g) \pi(g) v \d g = \tilde{h}(\mu_\pi)v.
\]

We define the \emph{Abel--Satake transform} (also known as \emph{Harish-Chandra transform}) to be the map $\calS:C_c(K_\infty \backslash G_\infty / K_\infty)\to C_c(T_\infty)$ defined by
\[
h\mapsto \calS h : a\mapsto \delta_\infty(a)^{1/2} \intop_{N_\infty}h(an)\d n,
\]
where $\d n$ is a Haar measure on $N_\infty$.
Since $\calS h$ is left-$K_\infty \cap T_\infty$-invariant, it is actually a map on $T_\infty^0$.
We have the exponential map $\exp\colon\a_\infty\to T_\infty^0$ which gives an identification of $\a_\infty$ with $T_\infty^0$. So we may as well realize $\calS h \in C_c(\a_\infty)$
after pre-composing with the $\exp$ map.

It holds that $\calS f$ is $W_\infty$-invariant. In fact, Gangolli showed that \[\calS\colon C_c^
\infty(K_\infty \backslash G_\infty / K_\infty) \to C_c^\infty(\a_\infty)^{W_\infty}\]
is an isomorphism of topological algebras; see \cite[3.21]{Duistermaat1979spectra} (also see \cite{gangolli1971plancherel}).
Let us denote the Fourier--Laplace transform $C_c(\a_\infty)\to C(\a_{\infty,\C}^*)$ by the map
\[h\mapsto \hat{h}:\mu\mapsto \intop_{\a_\infty} h(\alpha) e^{\langle\mu, \alpha\rangle}\d\alpha,\]
Harish-Chandra showed that the spherical transform $\tilde{h} = \widehat{\calS(h)}$. Here $\d\alpha$ denotes the Lebesgue measure inherited from the $\R$-vector space structure of $\a_\infty$.

\subsection{The Paley--Wiener Theorem}

Recall the Abel--Satake transform $\calS$ above. 
Gangolli (see \cite[eq.3.22]{Duistermaat1979spectra}) proved that if $h\in C_c^\infty(\a_\infty)^{W_\infty}$ is such that $\supp (h) \subset \{\alpha \in \a_\infty\mid\|\alpha\|\le b\}$, then 
\begin{equation}\label{supp-condition}
    \supp (\calS^{-1} h) \subset K_\infty \{\exp \alpha\mid\alpha \in \a_\infty, \|\alpha\|\le b\}K_\infty.
\end{equation}

We also record the classical Paley--Wiener theorem, which implies that for $h\in C_c^\infty(\a_\infty)^{W_\infty}$ with support in $\{\alpha \in \a_\infty\mid\|\alpha\|\le b\}$ we have
\begin{equation}\label{paley-wiener-thm}
    |\hat{h}(\mu)|\ll_{N,h} \exp(b\|\Re (\mu) \|)(1+\|\mu\|)^{-N}
\end{equation}
for every $N\ge 0$.

\section{A Bound on an Intertwining Operator}

In this section we prove \cref{thm:unramified} and provide one of the main ingredients of the proof of \cref{thm:BCF}.

First, we record the required estimates of the logarithmic derivatives of the normalizing factors and the normalized intertwining operators, denoted by properties (TWN+) and (BD), respectively.

\subsection{Property (TWN+)}\label{sec:twn+}

In \cite[Definition 3.3]{finis2017analytic-global}\footnote{The quantities $\Lambda(\pi_\infty;p^{\mathrm{sc}})$ and $\level(\pi,p^{\mathrm{sc}})$ are denoted by $\Lambda(\pi_\infty;\hat{M}_\alpha)$ and $\level(\pi_f;\hat{M}^+_\alpha)$, respectively, in this article, following \cite{finis2015limit}.}, the authors defined the \emph{Tempered Winding Number (TWN+) property, strong version}, of a reductive group. Property (TWN+) is a global property that concerns the size of the logarithmic derivative of $n_\alpha(\pi,s)$.
Below we record an implication of property (TWN+) which we will use in this article.

\begin{lemma}\label{twn+}
Let $G$ satisfy property (TWN+). Let $M\in\calL$ be proper, $\pi\in\Pi_2(M)$, and $\alpha \in \Sigma_M$. Let $K\subset K_f$ be open-compact and $\tau\in\widehat{K_\infty}$ so that $\calA^2_\pi(P)^{\tau,K}\neq \{0\}$ for some $P\in\calP(M)$.
%
Then it holds that
\[
\intop_{t_0}^{t_0+1} \left|\frac{n'_\alpha(\pi,it)}{n_\alpha(\pi,it)}\right| \d t \ll \log \left(1+|t_0|+\nu(\pi_\infty)+\nu(\tau)+\level(K)\right),
\]
for any $t_0\in\R$.
\end{lemma}

\begin{proof}
Note that $|n_\alpha(\pi,it)|=1$. Now the proof follows immediately by applying \cref{bound-of-level} and \cref{bound-of-infinitesimal-character} to the inequality in \cite[Definition 3.3]{finis2017analytic-global}.
\end{proof}

Finally, we keep a record that property (TWN+) is not vacuous.

\begin{prop}\label{gln-twn}\cite[Theorem 3.11]{finis2017analytic-global}.
Property (TWN+) holds for $G$ being $\GL_n$ and its inner form, a quasi-split classical group\footnote{The proof of \cite{finis2017analytic-global} for the classical groups crucially depends on the endoscopic classification of automorphic representations, due to Arthur \cite{arthur2013endoscopic} and Mok \cite{mok2015endoscopic}.}, and $G_2$. Moreover, it holds for any group whose derived group coincides with the derived group of any group mentioned above.
\end{prop}

\subsection{Property (BD)}\label{sec:bd}

In \cite[Definition 1-2]{finis2019analytic-local}, the authors defined the \emph{Bounded Degree (BD) property} of a reductive group. Property (BD) is a local property that concerns the size of the logarithmic derivative of $R_{Q\mid P}(\pi_v,s)$. Below we record an implication of property (BD) which we will use in this article.

\begin{lemma}\label{bd}
Let $G$ satisfy property (BD) at every place. Let $M\in\calL$ be proper, $\pi\in \Pi_2(M)$, and $P,Q\in\calP(M)$ be adjacent. Let $K\subset K_f$ be open-compact and $\tau\in\widehat{K_\infty}$ so that $\calA^2_\pi(P)^{\tau,K}\neq \{0\}$. Then it holds that
\[
\intop_{t_0}^{t_0+1} \left\|R_{Q|P}(\pi,it)^{-1} R_{Q|P}^\prime(\pi,it)\vert_{\left(\ind_{P(\bbA)}^{G(\bbA)}\pi\right)^{\tau,K}}\right\|_{\op} \d t \ll \log \left(1+\nu(\tau)+\level(K)\right),
\]
for any $t_0\in\R$.
\end{lemma}

To prove \cref{bd}, we use a different formulation of property (BD), namely, \cite[Definition 5.9]{finis2015limit}. This formulation is implied by the formulation in \cite[Definition 1-2]{finis2019analytic-local}; see \cite[Remark 3]{finis2019analytic-local}.

\begin{proof}
We closely follow the proof of \cite[Proposition 5.16]{finis2015limit}. As $R_{Q\mid P}(\pi,s)=\prod_v R_{Q\mid P}(\pi_v,s)$ and as $R_{Q\mid P}(\pi_v,s)\vert_{\left(\ind_{P(F_v)}^{G(F_v)}\pi_v\right)^{K_v}}$ is independent of $s$ for almost all $v$ (see \cite[Remark 5.13]{finis2015limit}), it is enough to show that
\[\intop_{t_0}^{t_0+1} \left\|R_{Q|P}(\pi_\infty,it)^{-1} R_{Q|P}^\prime(\pi_\infty,it)\vert_{\left(\ind_{P(F_\infty)}^{G(F_\infty)}\pi_\infty\right)^{\tau}}\right\|_{\op} \d t \ll \log \left(1+\nu(\tau)\right),\]
and for each non-archimedean $v$
\[\intop_{t_0}^{t_0+1} \left\|R_{Q|P}(\pi_v,it)^{-1} R_{Q|P}^\prime(\pi_v,it)\vert_{\left(\ind_{P(F_v)}^{G(F_v)}\pi_v\right)^{K_v}}\right\|_{\op} \d t \ll \log \left(1+\level(K_v)\right)\]
where $K=\prod_{v<\infty}K_v$ and $\level(K_v)$ is the $v$-adic factor of $\level(K)$ and the implied constant is uniform in $v$ (\emph{i.e.} uniform in the order $q_v$ of the corresponding residue fields).

To prove the non-archimedean part we only modify \cite[eq.(14)]{finis2015limit}. Keeping the same notations (in particular, denoting $R_{Q\mid P}(\pi_v,s)\vert_{\left(\ind_{P(F_v)}^{G(F_v)}\pi_v\right)^{K_v}} = A_v(q_v^{-s})$) and assumptions as in \cite{finis2015limit} we bound
\begin{align*}
    \intop_{t_0}^{t_0+1} 
    &\left\|R_{Q|P}(\pi_v,it)^{-1} R_{Q|P}^\prime(\pi_v,it)\vert_{\left(\ind_{P(F_v)}^{G(F_v)}\pi_v\right)^{K_v}}\right\|_{\op} \d t\\
    &=\log q_v\intop_{t_0}^{t_0+1} \|A'_v(q_v^{-it})\|_{\op} \d t \ll \log q_v\intop_{S^1}\|A'_v(z)\|_{\op}|\d z|,
\end{align*}
where the last estimate follows by changing variable $q_v^{-it}\mapsto z$. In the first equality above, the factor $\log q_v$ comes from
$$R'_{Q\mid P}(\pi_v,s) = -\log q_v A'_v(q_v^{-s}).$$
In the second estimate above, we have changed variable $$q_v^{-it}\mapsto z,\quad \frac{\d z}{z} = -i \log q_v\d t.$$
On the other hand, as $t$ moves along $[t_0,t_0+1]\subset \R$, the variable $z$ winds $S^1$ asymptotically $\log q_v$ times, uniformly in $t_0$.
From here we follow the discussion after \cite[eq.(14)]{finis2015limit} and conclude.

\vspace{5mm}

To prove the archimedean part we first modify \cite[Lemma 5.19]{finis2015limit} and its proof. Keeping the same notations and assumptions as in \cite{finis2015limit}, we claim that
\[\intop_{t_0}^{t_0+1}\|A'(it)\|_{\op}\d t \ll \sum_{1\le j\le m:\,u_j\neq 0} \min \left(1,\frac{1}{|u_j|}\right).\]
Here $A(s)=A_\infty(s)=R_{Q\mid P}(\pi_\infty,s)\vert_{\left(\ind_{P(F_\infty)}^{G(F_\infty)}\pi_\infty\right)^{\tau}}$.
Following the proof of \cite[Lemma 5.19]{finis2015limit}, we see that the above claim follows from the following claim
\[
\intop_{t_0}^{t_0+1} |\phi_w'(it)| \d t \ll \min \left( 1, 1/|u|\right),
\]
where $w=u+iv$ with $u,v\in \R$,
\[\phi_w(z):=\frac{z+\overline{w}}{z-w},\quad |\phi_w'(it)| = \frac{2|u|}{u^2+(t-v)^2}.\]
If $u\neq 0$ we clearly have
\[\intop_{t_0}^{t_0+1}|\phi'_w(it)|\d t \ll \intop_{t_0}^{t_0+1}\frac{1}{|u|}\d t \ll \frac{1}{|u|}.\]
We also have
\[\intop_{t_0}^{t_0+1}|\phi'_w(it)|\d t \le \intop_{\R}\frac{2|u|}{u^2+(t-v)^2}\d t =\left.2\arctan\left(\frac{t-v}{|u|}\right)\right\vert^{\infty}_{-\infty} = 2\pi.\]
Thus the claim follows.

Now it follows from the discussion and the proof in \cite[pp.620-621]{finis2015limit} that 
\[\intop_{t_0}^{t_0+1} \left\|R_{Q|P}(\pi_\infty,it)^{-1} R_{Q|P}^\prime(\pi_\infty,it)\vert_{\left(\ind_{P(F_\infty)}^{G(F_\infty)}\pi_\infty\right)^{\tau}}\right\|_{\op} \d t \ll \sum_{j=1}^r\sum_{k=1}^m\min\left(1,\frac{1}{|u_j-ck|}\right),\]
where $|u_j|,r\ll_G 1$,
\[m\ll 1+\|\tau\|\, \asymp 1+\nu(\tau)^2,\]
and $c>0$ depends only on $M$. We conclude the proof in the archimedean case by noting that the double sum above is bounded by $\ll_{G,M}\log m$.

\end{proof}

Finally, we keep a record that the property (BD) is not vacuous.

\begin{prop}\label{gln-bd}\cite[Corollary 1]{finis2019analytic-local}.
Property (BD) holds at every place if $G$ is $\GL_n$ or its inner form, or a split group of rank $2$. It also holds at every non $2$-adic place if $G$ is special orthogonal, or symplectic.
\end{prop}

\vspace{0.5cm}

Fix $T\in\a_{M_0}^G$. Let $M\in \calL$ and fix $P\in \calP(M)$. For $Q\in \calP(M)$ and
$\lambda,\Lambda \in i(\a_M^G)^*$, 
we define an operator $\calM_Q^T(\lambda,\Lambda,P)$ acting on $\calA^2(P)$ by
\begin{equation}\label{def-calM}
    \calM_Q^T(\lambda,\Lambda,P) := e^{\langle \Lambda,Y_Q(T) \rangle }M_{Q|P}(\lambda)^{-1}M_{Q|P}(\lambda+\Lambda).
\end{equation}
Similarly, for $L\in \calL(M)$, $Q\in \calP(L)$ and $\lambda,\Lambda \in i(\a_L^G)^* \subset i(\a_M^G)^*$ we define 
\[\calM_Q^T(\lambda,\Lambda,P) := \calM_{Q_1}^T(\lambda,\Lambda,P),\]
for arbitrary $Q_1\subset Q$, $Q_1 \in \calP(M)$. The choice of $Q_1$ does not matter here. Here in the left hand side $\lambda,\Lambda$ are realized as restrictions on $\a_L$; see \cite[pp.133]{arthur2005intro}.

Following definitions in \cite[\S21]{arthur2005intro}, the set $\{ \calM_Q^T(\lambda,\Lambda,P) \}_{Q\in \calP(L)}$ is an example of a ``$(G,L)$-family"; see \cite[\S17]{arthur2005intro} for a general discussion on $(G,M)$-family.
For $L\in \calL(M)$ and $\lambda \in i(\a_L^G)^*$,
we define the operator $\calM^T_L(\lambda,P)$ as the value at $\Lambda=0$ of
\[
\sum_{Q \in \calP(L)}\frac{\calM_Q^T(\lambda,\Lambda,P)}{\theta_Q(\Lambda)}.
\]
Below we prove the main estimate on the operator norm of $\calM^T_L(\lambda,P)$.

\begin{prop}\label{prop:calM bound}
Assume that $G$ satisfies properties (TWN+) and (BD). Fix $M\in\calL$, $P\in\calP(M)$, and $L\in\calL(M)$. Let $\tau\in\widehat{K_\infty}$, $K\subset K_f$ be an open-compact, and $\pi\in\Pi_2(M)$. Then for every $\lambda_0 \in i(\a_L^G)^*$,
\[
    \intop_{\substack{\lambda\in i(\a_L^G)^* \\ \|\lambda-\lambda_0\|\le 1}}\left\|\calM^T_L(\lambda,P)|_{\calA^2_\pi(P)^{\tau,K}} \right\|_{\op} \d \lambda \ll \left(\|T\|\,\log \left(1+\nu(\tau)+\nu(\pi_\infty)+\|\lambda_0\|+\level(K)\right)\right)^{\dim \a_L^G},
\]
where $T\in\a_{M_0}^G$ with sufficiently large $d(T)$.
\end{prop}

Before proving \cref{prop:calM bound}, we give a proof of \cref{thm:unramified}.

\begin{proof}[Proof of \cref{thm:unramified}]
Using \cref{lem:Maass-Selberg unram} and the fact that $M_{Q\mid P}(\lambda)$ unitary for $\lambda\in i(\a_M^G)^*$, and recalling \cref{def-calM}, 
we can write
\[
\|\Lambda^T \Eis(\varphi,\lambda)\|_2^2 = \langle \calM^T_M(\lambda,P) \varphi, \varphi \rangle_{\calA^2(P)}.
\]
We conclude using \cref{prop:calM bound} and that $\varphi$ is a unit.
\end{proof}

The rest of this section is devoted to the proof of \cref{prop:calM bound}.

Let $P|^\alpha Q$ for some $\alpha\in\Sigma_M$. Let $\varpi \in \a_M^*$ be such that $\langle \varpi,\alpha^\vee\rangle=1$. We define 
\begin{equation}\label{defn-delta}
    \delta_{Q|P}(\lambda) := M_{Q|P}^{-1}(\lambda)D_{\varpi}M_{Q|P}(\lambda)
\end{equation}
as in \cite[\S 4.3]{finis2015limit} as an operator from $\calA^2(P)$ to $\calA^2(P)$. Here $D_{\varpi}M_{Q|P}(\lambda)$ is the derivative of $M_{Q|P}(\lambda)$ in the direction $\varpi$. The definition is independent of the choice of $\varpi$; see \cite[pp.608]{finis2015limit}.

\begin{lemma}\label{lem:delta bound}
Keep the notations as above and assumptions as in \cref{prop:calM bound}.
We have
\[
\intop_{t_0}^{t_0+1} \left\| \delta_{Q|P}(it\varpi)|_{\calA^2_{\pi}(P)^{\tau,K}} \right\|_{\op} \d t \ll
    \log\left(1+\nu(\tau)+\nu(\pi_\infty)+|t_0|+\level(K)\right),
\]
for any $t_0\in\R$.
\end{lemma}

\begin{proof}
We take inverse on the both sides of \cref{factorization-intertwiner} and compose it with the derivative of \cref{factorization-intertwiner}, and obtain
\[
j_P^{-1}\circ \delta_{Q|P}(s\varpi) \circ j_P = \frac{n'_\alpha(\pi,s)}{n_\alpha(\pi,s)}\mathrm{Id}+R_{Q|P}^{-1}(\pi,s)R_{Q|P}^\prime(\pi,s).
\]
Recall that $j_P$ can be chosen to be isometric, i.e., $\|j_P\|=1$. We conclude the proof by applying \cref{twn+} and \cref{bd}.
\end{proof}

Given $L \in \calL(M)$, we let $\frakB_{P,L}$ be the set of tuples $\underline{\beta} = (\beta_1^\vee,...,\beta_m^\vee)$ of $\Sigma_P^\vee$ with $m:=\dim\a^G_L$ such that their projection to $\a_L^G$ form a basis of $\a_L^G$. We denote $\vol(\underline{\beta})$ to be the co-volume of the lattice generated by the projection of $\beta_1^\vee,...,\beta_m^\vee$ in $\a_L^G$. 
Let 
\[
\Xi_L(\underline{\beta}):=\{(P_1,...,P_m,P_1^\prime,...,P_m^\prime) \in \calP(M)^{2m} \mid P_i,P_i^\prime \in \calP(M), P_i|^{\beta_i}P_i^\prime\}.
\]
In \cite[pp.179-180]{finis2011spectral-abs}, (also see \cite[pp.608-609]{finis2015limit}) the authors define a map\footnote{The definition of $\xi$ also depends on $L$ and on a certain $\mu\in\a_M^*$, but we do not need the precise definition of the map.} \[\frakB_{P,L}\ni\underline{\beta}\mapsto \xi(\underline{\beta})\in\Xi_L(\underline{\beta}).\]
For $P|^\beta P'$ we find $\alpha_{P,P^\prime}\in \a_M^*$ such that
\[
Y_{P}(T)-Y_{P'}(T) = \langle \alpha_{P,P^\prime},T-T_0 \rangle \beta^\vee.
\]
Such an $\alpha_{P,P'}$ exists; see \cite[pp.186]{finis2011spectral-abs}.
Finally, for $\xi\in \Xi_L(\underline{\beta})$ we define the operator 
\begin{multline}\label{eq:Delta}
    \Delta_{\xi}^T(P,\lambda) := \frac{\vol(\underline{\beta})}{m!}M_{P\mid P_1^'}(\lambda) \left(\delta_{P_1|P_1^\prime}(\lambda)+\langle\alpha_{P_1,P_1^\prime},T-T_0\rangle\mathrm{Id}\right) M_{P_1^\prime|P_2^\prime}(\lambda)\\
    \dots \left(\delta_{P_m|P_m^\prime}(\lambda)+\langle\alpha_{P_m,P_m^\prime},T-T_0\rangle\mathrm{Id}\right)M_{P_m^\prime|P}(\lambda)
\end{multline}
Note that the definition of $\Delta_\xi^T(P,\lambda)$ here looks a little different from \cite[pp.186]{finis2011spectral-abs}, although they are the same. This is because the definition of $\delta_{P_i\mid P_i'}$ in \cite[pp.179]{finis2011spectral-abs} is different than ours (also compare with the definition of $\Delta^T_\xi(P,\lambda)$ in \cite[pp.609]{finis2015limit}).

Following \cref{prop:FLM formula} expresses the operator $\calM^T_L$ in terms of $\Delta^T_\xi$. A proof of this is essentially contained in \cite[Theorem 4]{finis2011spectral-abs}.

\begin{prop}\label{prop:FLM formula}
Keep the same notations as in \cref{prop:calM bound}. We have
\[
\calM_L^T(P,\lambda) = \sum_{\underline{\beta}\in \mathfrak{B}_{P,L}} \Delta_{\xi(\underline{\beta})}^T(P,\lambda)
\]
as operators on $\calA^2_{\pi}(P)^{\tau,K}$.
\end{prop}

\begin{proof}
We closely follow the proof of \cite[Theorem 4]{finis2011spectral-abs}. The Taylor expansion $\calA_Q$ of 
\[M_{Q\mid P}(\lambda)^{-1}M_{Q\mid P}(\lambda+\Lambda),\quad \text{at }\Lambda=0\]
form a compatible family (in the sense of \cite{finis2011spectral-abs}), when restricted to the finite-dimensional subspace $\calA_\pi^2(P)^{\tau,K}$. This follows from the argument in \cite[pp.185]{finis2011spectral-abs} for $s=1$. On the other hand, the Taylor series $c_Q$ of $e^{\langle\Lambda,Y_Q(T)\rangle}$ at $\Lambda=0$ form a scalar-valued compatible family. Thus following the argument in \cite[pp.186]{finis2011spectral-abs} we conclude that $c_Q\calA_Q$ form a compatible family when restricted to the finite-dimensional subspace $\calA_\pi^2(P)^{\tau,K}$. Hence, as in \cite{finis2011spectral-abs}, we compute the limit of \cref{def-calM} at $\Lambda=0$ using \cite[Theorem 8.1]{finis2011spectral-combinatorial}, which concludes the proof.
\end{proof}

Finally, we can prove \cref{prop:calM bound}.

\begin{proof}[Proof of \cref{prop:calM bound}]
Using \cref{prop:FLM formula} we see that it is enough to bound for each $\beta \in \frakB_{P,L}$,
\[
\intop_{\substack{{\lambda\in i(\a_L^G)^*} \\ \|\lambda-\lambda_0\|\le 1}}\left\| \Delta_{\xi(\underline{\beta})}^T(P,\lambda)|_{\calA^2_\pi(P)^{\tau,K}} \right\|_{\op} \d \lambda \ll \left(\|T\|\,\log \left(1+\nu(\tau)+\nu(\pi_\infty)+\|\lambda_0\|+\level(K)\right)\right)^{m}. 
\]
Let $\vartheta_i\in(\a_L^G)^*$ be the dual basis to $\beta_j^\vee$.
Using this basis the above integral can be majorized by 
\[\ll \intop_{\substack{{\lambda=i\sum_{j=1}^mt_j\vartheta_j}\\{t_j\in\R,\,|t_j-t_{0,j}| \le  C}}}\left\|\Delta_{\xi(\underline{\beta})}^T(P,\lambda)|_{\calA^2_\pi(P)^{\tau,K}} \right\|_{\op} \d \lambda,\]
where $it_{0,j}:=\langle \lambda_0,\beta_j^\vee\rangle$ and $C$ is a constant depending only on the root system. Now we recall \cref{eq:Delta} and use the fact that $M_{Q|P}(\lambda)$ is unitary for purely imaginary $\lambda$ to see that the above integral can be bounded by
\[
\prod_{j=1}^m\intop_{|t_j - t_{0,j}| \le C} \left(\left\|\delta_{P_j|P_j^\prime}(i\vartheta_jt_j)|_{\calA^2_\pi(P'_j)^{\tau,K}}\right\|_{\op}+\|T\|\right) \d t_j.
\]
We conclude the proof by dividing the above interval in the $j$'th integral in $O(1)$ number of length one sub-interval centering the points of size $O(t_{0,j})$ and applying \cref{lem:delta bound} for the integral over each sub-interval.
\end{proof}

\section{The Spectral Side of the Trace Formula}

For any $M\in\calL$ and $\pi\in\Pi_2(M)$ we denote $\calB_{\chi,\pi}(P)$ to be an orthonormal basis of $\calA^2_{\chi,\pi}(P)$. Let $h\in C_c^\infty(G(\bbA)^1)$ be a $K_\infty K_f$-finite function and $T\in\a_{M_0}^G$ with sufficiently large $d(T)$.

Let $\chi$ be a cuspidal data. In \cite[\S3]{arthur1980trace-II}, Arthur defined a distribution $J_\chi^T(h)$ via a \emph{geometric} expansion (for $d(T)$ sufficiently large) and showed that $J_\chi^T(h)$ is a polynomial in $T$. We do not need to know the exact definition of the distribution $J_\chi^T(h)$ but will record \emph{spectral} formul{\ae} for it that will be useful for us.

We recall certain norm $\|g\|$ on $G(\bbA)\ni g$ as defined in \cite[pp.1252-1253]{arthur1982family-I}. For $P \supset P_0$ we let
\[n_P:=\sum_{P_0\subset P'\sim P}|W(\a_P,\a_{P'})|,\]
where $\sim$ denotes association.

\begin{prop}\label{prop:pre-trace}\cite[Proposition 2.2]{arthur1982family-I}
There is a constant $C_0>0$ such that if 
\[
d(T) \ge C_0 \left(1+\sup\{\log\|g\| \mid g\in \supp(h)\}\right),
\]
then
\[
J_\chi^T(h)= \sum_{P\supset P_0} \frac{1}{n_P}\sum_{\pi \in \Pi_2(M_P)} \sum_{\varphi \in \calB_{\chi,\pi}(P)}\intop_{i (\a_P^G)^*} \langle \Lambda^T \Eis(\rho(P,\lambda,h)\varphi,\lambda),\Lambda^T \Eis(\varphi,\lambda) \rangle_{\calA^2(G)} \d \lambda,
\]
and moreover, is a polynomial in $T$.
\end{prop}

The $\pi$-sum in \cref{prop:pre-trace} is actually finite
(see \cite[pp.1245]{arthur1982family-I}), although we will not use this fact in this paper.

The polynomial expression of $J_\chi^T(h)$ can be explicitly calculated. We record the expression in the following \cref{prop:spectral side}.

For $L\in\calL(M)$, let $W^L(M_P)$ be the set of $w\in W(M)$ such that
\[\{H\in\a_M \mid wH=H\} = \a_L.\]
We record the following important proposition. For each $w\in W^L(M_P)$ we consider the intertwining operator $M(P,w)$ defined in \cite[pp.1309]{arthur1982family-II}. We do not need an explicit definition (which is a bit convoluted) of this operator but only need some of its properties which is described below.

\begin{prop}\label{prop:spectral side}
The quantity $J_\chi^T(h)$ is equal to
\[\sum_{P\supset P_0} \frac{1}{n_P}\sum_{\pi \in \Pi_2(M_P)}\sum_{L\in\calL(M_P)}\sum_{w\in W^L(M_P)}\iota_w \intop_{i(\a_{L}^G)^*} \tr \left(\calM^T_{L}(P,\lambda)M(P,w)\rho(P,\lambda,h)|_{\calA^2_{\chi,\pi}(P)}\right) \d \lambda.\]
where $\iota_w:=\left|\det\left((w-1)\mid_{\a^L_M}\right)\right|$.
\end{prop}

Arthur in \cite{arthur1982family-I,arthur1982family-II} proved \cref{prop:spectral side} only for $T=T_0$, but below we show that the same arguments can also be made to work for general $T$. 

\begin{proof}
We fix a Weyl group invariant test function $B\in C_c^\infty(i(\a_{M_0}^G)^*)^{W_0}$ with $B(0)=1$. As in \cite[pp.1273]{arthur1982family-I}, for $\epsilon>0$ and $\pi\in\Pi_2(M)$ we define
\[B^\epsilon_\pi(\lambda):=B\left(\epsilon(i\Im(\mu_\pi)+\lambda)\right),\quad\lambda\in i(\a_L^G)^*,\]
where $\mu_\pi\in \a_{M_0,\C}^*$ is the infinitesimal character of $\pi$. Combining \cite[Theorem 6.3(ii)]{arthur1982family-I} and \cite[Theorem 4.1]{arthur1982family-II} we obtain that $J_\chi^T(h)$ is equal to the limit of
\[\sum_{P\supset P_0} \frac{1}{n_P}\sum_{L\in\calL(M_P)}\sum_{w\in W^L(M_P)}\iota_w \sum_{\pi \in \Pi_2(M_P)}\intop_{i(\a_{L}^G)^*} 
\tr \left(\calM^T_{L}(P,\lambda)M(P,w)\rho(P,\lambda,h)|_{\calA^2_{\chi,\pi}(P)}\right) B_\pi^\epsilon(\lambda) \d \lambda\]
as $\epsilon\to 0$. Now if we show that for each $P\in\calP$, $L\in\calL(M_P)$, and $w\in W^L(M_P)$
\[\sum_{\pi \in \Pi_2(M_P)}\intop_{i(\a_{L}^G)^*} 
\left|\tr \left(\calM^T_{L}(P,\lambda)M(P,w)\rho(P,\lambda,h)|_{\calA^2_{\chi,\pi}(P)}\right)\right|\d \lambda<\infty\]
then by dominated convergence, we can pass the $\epsilon\to 0$ limit inside and conclude.

Now we work as in \cite[\S5.1]{finis2011spectral-abs} to show the above absolute convergence. As the intertwiner $M(P,w)$ commutes with the group action $\rho$ and is unitary (see \cite[pp.184]{finis2011spectral-abs} we have
\[\tr \left(\calM^T_{L}(P,\lambda)M(P,w)\rho(P,\lambda,h)|_{\calA^2_{\chi,\pi}(P)}\right)\le \left\|\calM^T_{L}(P,\lambda)\rho(P,\lambda,h)|_{\calA^2_{\chi,\pi}(P)}\right\|_1.\]
Here on the right hand side above and elsewhere, $\|\cdot\|_1$ denotes the trace norm of the corresponding operator.
Moreover, using \cref{prop:FLM formula} it suffices to show that
\[\sum_{\pi \in \Pi_2(M_P)}\intop_{i(\a_{L}^G)^*}\left\|\Delta_\xi^T(P,\lambda)\rho(P,\lambda,h)|_{\calA^2_{\chi,\pi}(P)}\right\|_1 \d \lambda <\infty,\]
for any $\underline{\beta}\in\calB_{P,L}$ and $\xi\in\Xi_L(\underline{\beta})$.
As $h$ is $K_f$-finite, there exists an open-compact $K_h\subset K_f$ such that $h$ is right $K_h$-invariant. Now decomposing into $K_\infty$-types it suffices to show that
\[\sum_{\tau\in\widehat{K_\infty}}\sum_{\pi \in \Pi_2(M_P)}\intop_{i(\a_{L}^G)^*}\left\|\Delta_\xi^T(P,\lambda)\rho(P,\lambda,h)|_{\calA^2_{\chi,\pi}(P)^{\tau,K_h}}\right\|_1 \d \lambda <\infty.\]

Working, as in \cite[\S5.1]{finis2011spectral-abs}, with the Laplacian on $G(F_\infty)$, noting that $h$ is compactly supported, and using $\|A\mid_V\|_1\le\dim(V)\|A\|_{\op}$, we obtain the above integral is bounded by
\[\ll_{h,N} \dim\left(\calA_{\chi,\pi}^2(P)^{\tau,K_h}\right)\intop_{i(\a_{L}^G)^*}(1+\|\lambda\|+\nu(\pi_\infty)+\nu(\tau))^{-N}\left\|\Delta^T_{\xi}(P,\lambda)\mid_{\calA^2_{\chi,\pi}(P)^{\tau,K_h}}\right\|_{\op}\d\lambda,\]
for any $N>0$. Now we use \cref{eq:Delta} and work as in the proof of \cref{prop:calM bound} to bound
\[\left\|\Delta^T_{\xi}(P,\lambda)\mid_{\calA^2_{\chi,\pi}(P)^{\tau,K_h}}\right\|_{\op}\ll_T \prod_{j}
\left(1+\left\|\delta_{P_j|P_j^\prime}(i\vartheta_jt_j)|_{\calA^2_\pi(P'_j)^{\tau,K_h}}\right\|_{\op}\right),
\]
where, following the notation in the proof of \cref{prop:calM bound}, above we have written $\lambda=i\sum_jt_j\vartheta_j$ with $\vartheta_j\in(\a_L^G)^*$.
Finally, from here on proceeding as in \cite[\S5.1]{finis2011spectral-abs} we conclude the proof.
\end{proof}

\vspace{1cm}

We denote $\calA^2_{\pi}(P):=\bigoplus_{\chi}\calA^2_{\chi,\pi}(P)$ and $\calB_\pi(P):=\bigcup_\chi\calB_{\chi,\pi}(P)$, which is an orthonormal basis of $\calA^2_\pi(P)$. On the other hand, the sum
\begin{equation}\label{defn-JT}
    J^T(h):=\sum_{\chi}J_\chi^T(h),
\end{equation}
where $\chi$ runs over all equivalence classes of cuspidal data, is absolutely convergent; see \cite[pp.80]{arthur2005intro}.

\begin{prop}\label{prop:spectral-side-full}
Let $T$ be as in \cref{prop:pre-trace}. Then
\begin{align*}
&J^T(h)= \sum_{P\supset P_0} \frac{1}{n_P}\sum_{\pi \in \Pi_2(M_P)} \sum_{\varphi \in \calB_{\pi}(P)}\intop_{i (\a_P^G)^*} \langle \Lambda^T \Eis(\rho(P,\lambda,h)\varphi,\lambda),\Lambda^T \Eis(\varphi,\lambda) \rangle_{\calA^2(G)} \d \lambda\\
&=\sum_{P\supset P_0} \frac{1}{n_P}\sum_{\pi \in \Pi_2(M_P)}\sum_{L\in\calL(M_P)}\sum_{w\in W^L(M_P)}\iota_w \intop_{i(\a_{L}^G)^*} \tr \left(\calM^T_{L}(P,\lambda)M(P,w)\rho(P,\lambda,h)|_{\calA^2_{\pi}(P)}\right) \d \lambda,
\end{align*}
where $\iota_w$ and $M(P,w)$ are as in \cref{prop:spectral side}.
\end{prop}

\begin{proof}
The proofs are similar (almost verbatim) to those of \cref{prop:pre-trace} and \cref{prop:spectral side}, after replacing $\calA^2_{\chi,\pi}$ and $\calB_{\chi,\pi}$ with $\calA^2_\pi$ and $\calB_\pi$, respectively.
\end{proof}

\section{Bounds for a specific Maximal Compact}

In this section we prove \cref{thm:BCF}. 
First, we develop some necessary local ingredients that go in the proof.

\subsection{Choice of test function}\label{sec:constructing-test-function}
We freely use the notations of \cref{subsec:spherical function}. Recall that $\a_\infty$ is the Lie algebra of a maximal $\R$-split torus of $G(F_\infty)$. 
Let $\delta>0$ be a sufficiently small constant. We choose a fixed $L^1$-normalized non-negative test function $f_{0,\delta}$ on $\a_\infty$ supported on the ball of radius $\delta/2$ around $0$. Let $f_{\delta}:=f_{0,\delta}*f_{0,\delta}^*$ where for any $f\in C_c(\a_\infty)$ we define $f^*(\alpha) := \overline{f(-\alpha)}$. It holds that $\widehat{f^*}(\mu) = \overline{\hat{f}(-\overline{\mu})}$. Thus, $\widehat{f_{\delta}}$ is non-negative on $i\a_\infty^*$. Note that $f_{\delta}$ is also non-negative, $L^1$-normalized, and supported on the ball of radius $\delta$ around $0$. 

We fix
\begin{equation}\label{eq:value-of-c0}
    C_0:=2\|\rho_\infty\|+1.
\end{equation}
Now we make $\delta>0$ small enough so that
\begin{equation}\label{eq:localization-continuity}
    \Re(\widehat{f_\delta}(\mu))\ge 1/2,\quad\text{ for }\|\mu\|\le C_0.
\end{equation}
This follows from continuity and the fact that $\widehat{f_\delta}(0)=1$.

Given $\mu_0 \in i\a_\infty^*$, we define $f_{\mu_0} \in C_c^\infty(\a_\infty)^{W_\infty}$ by
\[
f_{\mu_0}(\alpha) := \sum_{w \in W_\infty} f_\delta(w\alpha)e^{-\langle \mu_0,w\alpha \rangle},
\]
so that
\[\widehat{f_{\mu_0}}(\mu)=\sum_{w\in W_\infty}\widehat{f_\delta}(w\mu-\mu_0),\quad \mu\in\a^*_{\infty,\C}.\]
Finally, we define
\[h_{\mu_0} := \calS^{-1} (f_{\mu_0}^* *f_{\mu_0})\in C_c^\infty(K_\infty\backslash G(F_\infty)/K_\infty).\]
Thus the spherical transform of $h_{\mu_0}$ is \[\tilde{h}_{\mu_0}(\mu)=\widehat{f_{\mu_0}}(\mu)\overline{\widehat{f_{\mu_0}}(-\overline{\mu})};\]
see \cref{sec:spherical-transform}.
Then $h_{\mu_0}$ satisfies the following properties:
\begin{enumerate}
    \item\label{supp-h} The support of $h_{\mu_0}$ is bounded independently of $\mu_0$, which follows from \cref{supp-condition}.
    \item\label{nonnegative-h} $\tilde{h}_{\mu_0}(\mu) \ge 0$, if $\mu = -w\overline{\mu}$ for some $w\in W_\infty$; in particular, if $\mu\in i\a_\infty^*$ or more generally, if $\mu$ is the Langlands parameter of a unitary representation. This is immediate from $W_\infty$-invariance of $\widehat{f_{\mu_0}}$.
    \item\label{rapid-decay-h} There is a constant $b$ depending only on the support of $f_\delta$ such that for every $N>0$ it holds that 
    \[
    |\tilde{h}_{\mu_0}(\mu)| \ll_{N} \sum_{w \in W_\infty}\exp(b\|\Re(\mu)\|)(1+\|w\mu-\mu_0\|)^{-N},
    \]
    which follows from \cref{paley-wiener-thm}.
    \item\label{localization-h} For $C_0$ as in \cref{eq:value-of-c0} we have that if $\|\mu -\mu_0\| \le C_0$ then $|\tilde{h}_{\mu_0}(\mu)| \ge 1/10$, perhaps after making $\delta>0$ smaller. 
\end{enumerate}

To see the last claim it suffices to show that
\[\Re\left(\widehat{f_{\mu_0}}(\mu)\right) \ge 1/3,\quad\text{ for }\|\mu-\mu_0\|=\|-\overline{\mu}-\mu_0\|<C_0.\]
Working similarly to the proof of \cite[Lemma 3.2]{marshall2014upper}, below we prove the above sufficient formulation. We see that
\begin{align*}
    \widehat{f_\delta}(\mu)
    &=\intop_{\a_\infty}f_\delta(\alpha)\left(e^{i\langle\Im(\mu),\alpha\rangle}-e^{i\langle\Im(\mu),\alpha\rangle}+e^{\langle\mu,\alpha\rangle}\right)\d\alpha\\
    &=\widehat{f_\delta}\left(i\Im(\mu)\right) + \intop_{\a_\infty}f_\delta(\alpha)e^{i\langle\Im(\mu),\alpha\rangle}\left(e^{\langle\Re(\mu),\alpha\rangle}-1\right)\d\alpha.
\end{align*}
As $\|\Re(\mu)\|\le C_0$ we have $|e^{\langle\Re(\mu),\alpha\rangle}-1|<C_1\delta$ for $\alpha$ lying in a $\delta$-radius ball around the origin and for certain $C_1$ depending on $C_0$. Consequently, we have
\[\Re(\widehat{f_\delta}(\mu)) \ge - C_1\delta,\]
as $\widehat{f_\delta}\left(i\Im(\mu)\right)\ge 0$.
Thus upon applying \cref{eq:localization-continuity}, and the fact that 
\[\|\Re(w\mu-\mu_0)\|= \|\Re(\mu)\|=\|\Re(\mu-\mu_0)\|\le \|\mu-\mu_0\|\le C_0\] 
we obtain
\begin{equation*}
    \Re(\widehat{f_{\mu_0}}(\mu))
    =\Re(\widehat{f_\delta}(\mu-\mu_0))+\sum_{w\neq 1}\Re(\widehat{f_\delta}(w\mu-\mu_0))
    \ge 1/2-C_2 \delta,
\end{equation*}
for some $C_2$ depending on $C_0$ and the group. We conclude by making $\delta$ sufficiently small.

\subsection{Langlands parameters}

Let $\chi:=(M_1,\pi_1)$ be a cuspidal datum, determined by a Levi subgroup $M_1$ and a representation $\pi_1 \in \Pi_2(M_1)$. Let $P_1 \supset P_0$ be such that $M_1 = M_{P_1}$. We assume that $\pi_{1,\infty}$ is spherical, i.e., has a non-zero $K_\infty \cap M_1(F_\infty)$-invariant vector, so we may associate to it a Langlands parameter $\mu_{\pi_1}$. 

We first notice that $\chi$ is actually an equivalence class, whose size by \cite[\S 12]{arthur2005intro}, is at most \[
n_{P_1} = \sum_{P_0 \subset P'\sim P_1} |W(\a_{P_1},\a_{P'})|.
\]
We let $(M_1,\pi_1),...,(M_k,\pi_k)$ be the elements in this equivalence class, with $k\le n_{P_1} \le n_{P_0}$, and $M_i = M_{P_i}$ for $P_i \supset P_0$, all associated with $P_1$. Notice that the Langlands parameter $\mu_{\pi_i}$ is in the $W_\infty$-orbit of $\mu_{\pi_1}$, so $\|\mu_{\pi_i}\|=\|\mu_{\pi_1}\|$.

Now, let $P\supset P_0$ and $\pi\in\Pi_2(M_P)$ so that $\calA_{\chi,\pi}^2(P)^{K_\infty K}\neq\{0\}$. It follows from the Langlands construction of the residual spectrum that $\pi$ is defined using iterated residues from an Eisenstein series of one of the $\pi_i$, $1\le i\le k$, for some $P_i \sim P$. 
Let $\rho_\pi \in (\a_{M_i}^M)^* \subset (\a_{M_0}^G)^* \subset (\a_\infty)^*$ be the point where the residual representation is defined by iterated residues (see e.g., \cite[\S2]{arthur1982inner}). 
Then $\mu_\pi = \mu_{\pi_i}+\rho_\pi$. The crucial point here is that $\rho_\pi$ is real and universally bounded (in fact, $\|\rho_\pi\|\le\|\rho_\infty\|$, where $\rho_\infty$ is as in \cref{subsec:spherical function}). Therefore, $\|\mu_{\pi}-\mu_{\pi_i}\|\le \|\rho_\infty\|$ and
\begin{equation}\label{lem:size-laplace-ev-cuspidal-data}
1+\|\mu_{\pi}\|\asymp 1+\|\mu_{\pi_i}\|.
\end{equation}
The above properties of $\rho_\pi$ follow from Langlands' construction of the residual spectrum; see e.g. \cite[pp.1152, Proof of Lemma 4.3]{muller2000singularities}.

Next, the Langlands parameter of $\Ind_{P(F_\infty)}^{G(F_\infty)}\left(\pi_\infty\otimes e^{\langle \lambda, H_{M_P}(\cdot) \rangle}\right)$ for $\lambda\in(\a_{P,\C}^G)^*$ is $\mu_{\pi,\lambda}:=\mu_\pi+\lambda$.\footnote{Here we identify $i(\a_P^G )^* \subset \a_{\infty,\C}^*$.} If $\lambda\in i(\a_P^G)^*$ then this induction is unitary. Correspondingly, it holds that
\begin{itemize}
    \item $\|\Re(\mu_{\pi,\lambda})\|\le \|\rho_{\infty}\|$, and 
    \item there is a $w\in W_\infty$ so that $-\overline{\mu_{\pi,\lambda}}=w\mu_{\pi,\lambda}$ (as a matter of fact, we may choose $w \in N_{M_P(F_\infty)}(T_\infty)/T_\infty \subset W_\infty$);
\end{itemize}
see, e.g., \cite[Proposition 3.4]{Duistermaat1979spectra}.

\vspace{0.5cm}

We now prove \cref{thm:BCF}.

\begin{proof}[Proof of \cref{thm:BCF}]



We use property (3) in \cref{sec:truncation} to assume that $T$ in \cref{thm:BCF} is such that $d(T)$ is sufficiently large (in terms of $G$ and $F$) but fixed.

We assume that the cuspidal datum $\chi:=(M,\pi_0)$ with the equivalence class $(M_1,\pi_1),\dots,(M_k,\pi_k)$ as above. Recall that $\|\mu_{\pi_0}\|\asymp \|\mu_{\pi_i}\|$.

Let $i\a_\infty \ni \mu_j := i \Im (\mu_{\pi_j})+\lambda_0$ for $j=1,\dots,k$. We construct $h_{\mu_j}$ as in \cref{sec:constructing-test-function} with the choice of $C_0$ given in \cref{eq:value-of-c0}.
Now let $h_j \in C_c^\infty(G(\bbA)^1)$ be of the form 
\[h_j(g_\infty,g_f) := h_{\mu_j}(g_\infty)\frac{\One_K(g_f)}{\vol(K)}.\]
Note that \cref{supp-h} implies that support of $h_j$ has bounded measure.

Clearly, for any $P\supset P_0$ and $\lambda\in i(\a_{P}^G)^*$ the operator $\rho(P,\lambda,h_j)$ projects on $\calA^2(P)^{K_\infty K}$. Moreover, for any $\pi\in \Pi_2(M_P)$ and $\varphi\in\calA^2_{\chi,\pi}(P)^{K_\infty K}$ we have
\[\rho(P,\lambda,h)\varphi = \tilde{h}_{\mu_j}(\mu_{\pi,\lambda})\varphi.
\]
Recall from the discussion about the Langlands parameters above that there exists $w\in W_\infty$ so that $w\mu_{\pi,\lambda}=-\overline{\mu_{\pi,\lambda}}$.
Thus using \cref{nonnegative-h} we obtain that \[\tilde{h}_{\mu_j}(\mu_{\pi,\lambda})\ge 0.\]

On the other hand, if $\|\lambda-\lambda_0\|\le 1$, and $P\supset P_0$ and $\pi\in\Pi_2(M_P)$ are such that $\calA_{\chi,\pi}^2(P)\neq\{0\}$ then by the discussion above there is $1\le j\le k$ such that
\begin{equation*}
\|\mu_{\pi,\lambda}-\mu_j\|\le \|\Re(\mu_{\pi_j})\| + \|\rho_\pi\|+ \|\lambda-\lambda_0\| \le 2\|\rho_\infty\|+1 =C_0
\end{equation*}
Hence, applying \cref{localization-h} it holds that
\[
\tilde{h}_{\mu_j}(\mu_{\pi,\lambda}) \ge 1/10.
\]

Now we fix $T\in\a_{M_0}^G$ with sufficiently large $d(T)$ and apply \cref{prop:pre-trace} to write
\[J_\chi^T(h_j) = \sum_{P\supset P_0}\frac{1}{n_P}\sum_{\pi\in\Pi_2(M_P)}\sum_{\varphi\in\calB_{\chi,\pi}(P)^{K_\infty K}}\intop_{i(\a_P^G)^*}\tilde{h}_{\mu_j}(\mu_{\pi,\lambda})\|\Lambda^T\Eis(\varphi,\lambda)\|^2_2\d\lambda.\]
Here $\calB_{\chi,\pi}(P)^{K_\infty K}$ is an orthonormal basis of $\calA_{\chi,\pi}^2(P)^{K_\infty K}$.
Using the above properties of $\tilde{h}_{\mu_j}$ we obtain that
\[
\sum_j J_\chi^T(h_j)\gg \sum_{P\supset P_0}\sum_{\pi\in\Pi_2(M_P)}\sum_{\varphi\in\calB_{\chi,\pi}(P)^{K_\infty K}}\intop_{\substack{{\lambda\in i(\a_P^G)^*} \\ \|\lambda-\lambda_0\|\le 1}}\|\Lambda^T\Eis(\varphi,\lambda)\|^2_2\d\lambda.\]
Note that the right-hand side above equals to the left-hand side of the estimate in \cref{thm:BCF}.
%

Hence, it remains to bound $J_\chi^T(h_j)$ for $j=1,\dots,k$. As the proof is the same for all $j$ without loss of generality we will bound for $j=1$ and write $h=h_1$.

Using \cref{prop:spectral side} we see that 
\[
J_\chi^T(h)\ll \sum_{P\supset P_0}\max_{L\in\calL(M_P)}\max_{w\in W^L(M_P)}\sum_{\pi\in\Pi_2(M_P)}\intop_{i(\a_{L}^G)^*} \tr\left(\calM_{L}^T(P,\lambda)M(P,w) \rho(P,\lambda,h)|_{\calA_{\chi,\pi}^2(P)}\right)\d\lambda.
\]
Recall that $\rho(P,\lambda,h)\mid_{\calA_{\chi,\pi}^2(P)}$ projects on $\calA_{\chi,\pi}^2(P)^{K_\infty K}$ and acts there by the scalar $\tilde{h}_{\mu_1}(\mu_\pi+\lambda)$. Noting that $M(P,w)$ is unitary and commutes with $\rho(P,\lambda,h)$ we get that the the integral on the right-hand side above is bounded by 
\[
\dim\left(\calA_{\chi,\pi}^2(P)^{K_\infty K}\right)\intop_{i(\a_{L}^G)^*}\tilde{h}_{\mu_1}(\mu_\pi+\lambda)
\left\|\calM_{L}^T(P,\lambda)|_{\calA_{\chi,\pi}^2(P)^{K_\infty K}}\right\|_{\op}\d \lambda.
\]
Applying \cref{rapid-decay-h} 
we can bound the inner integral in the above display by
\begin{equation}\label{eq:apply-item-3}
\ll_N\max_{w\in W_\infty}\intop_{i(\a_L^G)^*}\left(1+\|w\lambda+wi\Im(\mu_\pi)-\lambda_0-i\Im(\mu_{\pi_1})\|\right)^{-N}\left\|\calM_{L}^T(P,\lambda)|_{\calA_{\chi,\pi}^2(P)^{K_\infty K}}\right\|_{\op} \d \lambda.
\end{equation}
For each $w\in W_\infty$ we define 
\[\lambda_1:=\lambda_{1,w,\pi,\pi_1}:=w^{-1}\left(\lambda_0-i\Im(w\mu_\pi-\mu_{\pi_1})\right)\in i(\a_\infty)^*.\]
Using \cref{lem:size-laplace-ev-cuspidal-data} we obtain
\[\|\lambda_1\|\ll
\|\lambda_0\|+\|\mu_\pi\|+\|\mu_{\pi_1}\|\,\asymp 
\|\lambda_0\|+\sqrt{\nu(\pi_{1,\infty})}.\]
Thus for each $w$ we bound the the integral on the right-hand side of \cref{eq:apply-item-3} by
\begin{multline}\label{eq:w-integral}
    \intop_{i(\a_L^G)^*}\left(1+\|\lambda-\lambda_1\|\right)^{-N}\left\|\calM_{L}^T(P,\lambda)|_{\calA_{\chi,\pi}^2(P)^{K_\infty K}}\right\|_{\op} \d \lambda\\
    \le\sum_{m=1}^\infty\frac{1}{m^N}\intop_{{\substack{{\lambda\in i(\a_L^G)^*} \\ m-1\le\|\lambda-\lambda_1\|\le m}}}\left\|\calM_{L}^T(P,\lambda)|_{\calA_{\chi,\pi}^2(P)^{K_\infty K}}\right\|_{\op} \d \lambda.
\end{multline}

We can find $\{\eta_l\}_{l=1}^{n}\subset i(\a_L^G)^*$ with  $\|\eta_l\|\le m$ and $n\ll m^{\dim\a_L^G}$ such that
\[\{\lambda:\|\lambda-\lambda_1\| \le m\}\subset \cup_{l=1}^n\{\lambda:\|\lambda-\lambda_1-\eta_l\|\le 1\}.\]
Thus the the integral on the right-hand side of \cref{eq:w-integral} is bounded by
\[\sum_{l=1}^n\intop_{{\substack{{\lambda\in i(\a_L^G)^*} \\ \|\lambda-\lambda_1-\eta_l\|\le 1}}}\left\|\calM_{L}^T(P,\lambda)|_{\calA_{\chi,\pi}^2(P)^{K_\infty K}}\right\|_{\op} \d \lambda.\]
We apply \cref{prop:calM bound} to estimate the above integral and obtain that 
\begin{multline*}
    \intop_{i(\a_{L}^G)^*}\tilde{h}_{\mu_1}(\mu_\pi+\lambda)
    \left\|\calM_{L}^T(P,\lambda)|_{\calA_{\chi,\pi}^2(P)^{K_\infty K}}\right\|_{\op} \d \lambda\\
    \ll_{N}\max_{w\in W_\infty}\sum_{m=1}^\infty\frac{1}{m^N}\sum_{l=1}^n\left(\|T\|\,\log\left(1+\|\lambda_1+\eta_l\|+\nu(\pi_\infty)+\level(K)\right)\right)^{\dim\a_L^G}
\end{multline*}
Using the bounds of $\lambda_1$, $\eta_l$, and $n$ the right hand side above is bounded by
\begin{multline*}
    \ll_{N}\sum_{m=1}^\infty\frac{1}{m^{N-\dim\a_L^G}}\left(\|T\|\,\log\left(1+m+\|\lambda_0\|+\nu(\pi_{1,\infty})+\level(K)\right)\right)^{\dim\a_L^G}\\
    \ll\left(\|T\|\,\log\left(1+\|\lambda_0\|+\nu(\pi_{1,\infty})+\level(K)\right)\right)^{\dim\a_L^G},
\end{multline*}
by making $N$ large enough. Thus we obtain
\begin{multline*}
    J_\chi^T(h) \ll \sum_{P\supset P_0}\max_{L\in\calL(M_P)}\sum_{\pi\in\Pi_2(M_P)}\dim\left(\calA_{\chi,\pi}^2(P)^{K_\infty K}\right)\\
    \times\left(\|T\|\,\log\left(1+\|\lambda_0\|+\nu(\pi_{1,\infty})+\level(K)\right)\right)^{\dim\a_L^G}.
\end{multline*}
We conclude noting that $\max_{L\in\calL(M_P)}\dim\a_L^G=\dim\a_M^G$.
\end{proof}

We now prove \cref{cor:BCF} and \cref{thm:gln-bcf}.

\begin{proof}[Proof of \cref{cor:BCF}]
By symmetries of the Eisenstein series, we may assume that $P\supset P_0$. 
Let $\chi = (M,\pi_0)$ be as in the statement of the corollary. Then by \cref{lem:size-laplace-ev-cuspidal-data} it holds that $1+\nu(\pi_{\infty}) \asymp 1+ \nu(\pi_{0,\infty})$ and the claim follows from \cref{thm:BCF}.
\end{proof}

\begin{proof}[Proof of \cref{thm:gln-bcf}]
We will prove the slightly stronger statement given in \cref{rem:level-gln}.

The classification of the residual spectrum for $\GL_n$ due to M{\oe}glin and Waldspurger \cite{moeglin1995spectral} (see \cref{sec:coarse-spectral-expansion}) implies that for every $\chi$ and $P\supset P_0$ it holds that
\[\#\{\pi\in\Pi_2(M_P)\mid \calA^2_{\chi,\pi}(P)\neq \{0\}\}\ll_n 1.\]

The classification also implies that the multiplicity of $\pi$ in 
$L^2_{\disc,\chi}(A_{M_P}M_P(F) \backslash M(\bbA))$ is at most $1$. So the representation $\rho(P,0,\cdot)$ on $\calA^2_{\chi,\pi}(P)$ is isomorphic to $\Ind_{P(\bbA)}^{G(\bbA)} \pi$, and in particular 
\[
\calA^2_{\chi,\pi}(P)^{K_\infty K} = (\Ind_{P(\bbA)}^{G(\bbA)} \pi)^{K_\infty K}
\]

On the other hand, 
we may write $\pi= \otimes_v^\prime \pi_v$,
and obtain
\[(\Ind_{P(\bbA)}^{G(\bbA)} \pi)^{K_\infty K} = \bigotimes{}^\prime(\Ind_{P(F_v)}^{G(F_v)} \pi_v)^{K_v},\]
where $K_\infty K = \prod_v K_v$.

Now using uniform admissibility due to Bernstein \cite{bernstein1974reductive}, which says that 
\[
\dim(\Ind_{P(F_v)}^{G(F_v)} \pi_v)^{K_v}\ll_{K_v} 1,
\]
we deduce that
\[
F(\chi;\triv,K) = \sum_{P\supset P_0}\sum_{\pi\in\Pi_2(M_P)}\dim \left(\calA^2_{\chi,\pi}(P)^{K_\infty K}\right) \ll_{n,K} 1.
\]

Let $\chi_0$ be a cuspidal datum and $\pi_0\in\Pi_2(M_P)$ such that $\varphi_0\in\calA^2_{\chi_0,\pi_0}(P)^{K_\infty K}$. Then we have $\nu(\varphi_0)\asymp\nu(\pi_{0,\infty})$. Now the result follows from \cref{cor:BCF} along with \cref{gln-twn} and \cref{gln-bd}.
\end{proof}

\section{Optimal Lifting after Assing--Blomer}

In this section, we prove \cref{thm:counting} and \cref{thm:optimal-lifting}. Aligning with our usual convention of suppressing the dependence on $G$ in Vinogradov's notations, we also suppress the dependence on $n$ in this section.

We start with a discussion of the link between the adelic setting of the group $\GL_n$ and the classical setting of $\SL_n(\R)$. A good place to follow is \cite[\S 3]{lapid2009spectral}.

From now on, let $G:=\GL_n$ over $\Q$. We identify $A_G$ as the positive scalar matrices in $G(\R)$. Note that $G(\R)/A_G \cong \SL_n(\R)\times \{\pm1\}$. 
We can therefore identify 
\[
G(\bbA)^1 = (G(\R)/A_G)\times \sideset{}{'}\prod_p G(\Q_p),
\]
where $\prime$ denotes the restricted product relative to $G(\Z_p)$.

Given an integer $q>0$ with prime factorization $q= \prod_{p<\infty} p^{r_p}$ with $r_p\in\Z_{\ge 0}$, let $K(q) \subset G(\bbA_f)$ be the principal congruence subgroup of level $q$. We have $K(q) = \prod_{p<\infty} K_p(q)$, where 
\[K_p(q) := \{g \in G(\Z_p) \mid g\equiv I \mod p^{r_p} \}.\]
The group $\SL_n(\R)$ acts from the right on the space $G(\Q)A_G\backslash G(\bbA)/ K(q)$. The space decomposes into $\varphi(q)$ ($\varphi$ being Euler's totient function) connected components, determined by the determinant map into 
\[
\R_{>0} \Q^\times \backslash \bbA^\times / \prod_{p\mid q} \{a\in \Z_p^\times \mid a\equiv 1 \mod p^{r_p}\Z_p \} \cong (\Z / q\Z)^\times. 
\]
Each orbit is isomorphic to $\Gamma(q) \backslash \SL_n(\R)$, where $\Gamma(q)$ is the principal congruence subgroup of level $q$ of $\Gamma(1):= \SL_n(\Z)$. 

Let $K_\infty^0 = \SO(n)$, which is a maximal compact subgroup of $\SL_n(\R)$, and an index $2$ subgroup of the maximal compact subgroup $K_\infty := \mathrm{O}(n)$ of $G(\R)$.
Let $\|\cdot\|$ denote the Frobenius norm on $\SL_n(\R)$, given by $g\mapsto \sqrt{\tr (g^\top g)}$. We note that the norm is both left and right $K_\infty^0$-invariant. 
It holds that
\[
\| g_1 g_2 \| \le \|g_1\| \| g_2 \|.
\]
For $R>n$ we define \[B_R:=\{g\in \SL_n(\R) \mid \| g \| \le R\}.\]
Note that
\begin{equation}\label{eq:vol of B_R}
    \vol(B_R) = c_n R^{n(n-1)}(1+o(1));
\end{equation}
see \cite[Appendix 1]{duke1993density}, which also calculates the precise coefficient $c_n$. 

We now construct a test function in the archimedean Hecke algebra following \cite[\S2.7]{assing2022density}. Let $h_0\colon \R_{\ge 0}\to [0,1]$ be a smooth non-negative function with $h_0(x)=1$ for $x\in[0,n]$ and $h_0(x)=0$ for $x>2n$. For $R>0$ we define
\[h_R\in C_c^\infty(K_\infty^0 \backslash \SL_n(\R) /K_\infty^0),\quad g\mapsto \frac{1}{\|h_0(\frac{\| \cdot\|}{R})\|_{L^1}}h_0\left(\frac{\| g\|}{R}\right).\]
A direct consequence of \cite[(2.34)-(2.35)]{assing2022density} is that
the spherical transform of $h_R$ satisfies
\begin{equation}\label{non-temp-paley-wiener}
    |\tilde{h}_R(\mu)|\ll_{N,\epsilon}(1+\|\mu\|)^{-N}R^{-\frac{n(n-1)}{2}+n\|\Re(\mu)\|_\infty+\epsilon}
\end{equation}
for $\mu\in\a^*_{0,\C}$ such that $w\mu=-\overline{\mu}$ for some $w\in W_\infty$ (and in particular, for any $\mu\in\a^*_{0,\C}$ corresponding to the Langlands parameters of a unitary representation). Here $\|\mu\|_\infty$ denotes the $\ell^\infty$-norm of $\mu$ when $\a^*_{0,\C}$ is realized as a complex $n$-tuple.

For the non-archimedean part we let $h_q\in C_c^\infty(G(\bbA_f))$ be the normalized characteristic function of $K(q)$, i.e.
\[
h_q := \frac{1}{\vol(K(q))}\One_{K(q)}. 
\]
Let $h_{R,q,1} \in C_c^\infty(G(\bbA)^1)$ be given by 
\[h_{R,q,1}(g_\infty,g_f) := h_R(g_\infty) h_q(g_f).\]
Finally, let $h_{R,q} = h_{R,q,1}*h_{R,q,1}^{*}$, where $h^*(g):=\overline{h(g^{-1})}$. In other words, we have
\[
h_{R,q}(g_\infty,g_f) = (h_R*h_R^*)(g_\infty)h_q(g_f)
\]
Clearly, $h_{R,q}\in C_c^\infty(G(\bbA)^1)$ and is bi-$K_\infty^0K_f$-finite.

The following is the main technical result of this section and will be used to prove \cref{thm:counting} and \cref{thm:optimal-lifting}. First, we recall the definition of $J^T(h)$ from \cref{defn-JT}.
\begin{prop}\label{prop:assing-blomer}
For any $q$ square-free and any $R>1$ it holds that for every $\epsilon>0$
\[
J^T(h_{R,q}) \ll_\epsilon \|T\|^{n-1}(Rq)^\epsilon (q+R^{-n(n-1)}q^{n^2}),
\]
for any $T\in\a_{M_0}$. 
\end{prop}

\begin{remark}
    As will be evident from the proof, the summand $q$ comes from the $\varphi(q)$ possible Dirichlet characters modulo $q$. The factor $R^{-n(n-1)}q^{n^2}$ comes from the contribution of cusp forms. The number of such forms with bounded Laplacian eigenvalues is $\asymp [K(1):K(q)] \ll q^{n^2}$ and the contribution of each such form to $J^T(h_{R,q})$ \emph{assuming the Generalized Ramanujan Conjecture (GRC)} is $R^{-n(n-1)}$. The results of Assing and Blomer \cite[Theorem 1.1]{assing2022density} on the \emph{Density Hypothesis} allow us to prove the proposition unconditionally of the GRC. However, a significant difficulty arises to control the continuous spectrum, as also experienced in \cite{assing2022density}, managing that is the main contribution of \cref{prop:assing-blomer}.
\end{remark}

We will assume \cref{prop:assing-blomer} for now and prove \cref{thm:counting} and \cref{thm:optimal-lifting}. 
In the next subsection, we prove \cref{prop:AB main prop} using \cref{prop:assing-blomer} and then deduce these two theorems from it in the following two subsections. At the end of this section, we prove \cref{prop:assing-blomer}.

\subsection{Initial reductions}

Let us fix $R_0>n$ and a compact domain $\Omega \subset G(\Q)\backslash G(\bbA)^1$ of the form $\Omega := G(\Q) \backslash B_{R_0}\times K(1)$. The implied constants in various estimates below are allowed to depend on $R_0$, however, we will not mention it in the notations.
We define the adelic kernel $K_{h_{R,q}}(x,y)$ by 
\[
K_{h_{R,q}}(x,y) := \sum_{\gamma \in G(\Q)} h_{R,q}(x^{-1}\gamma y),
\]
for $x,y \in G(\bbA)^1$.
\begin{prop}\label{prop: AB reduction1}
There is a constant $C$ (possibly depending on $R_0$) such that if $d(T) \ge C (1+\log (R))$ then it holds that 
\[
\intop_\Omega K_{h_{R,q}}(x,x) \d x \le J^T(h_{R,q}),
\]
for any $R$ and $q$.
\end{prop}

\begin{proof}
    We start with the spectral decomposition of $K_{h_{R,q}}(x,x)$ (see e.g., \cite[Equation (7.6)]{arthur2005intro}) which implies that
    \begin{align*}
    &\intop_{\Omega} K_{h_{R,q}}(x,x) \d x \\
    =& \sum_{P\supset P_0}\frac{1}{n_P} \sum_{\pi \in \Pi_2(M_P)} \sum_{\varphi  \in \calB_\pi(P)} \intop_{i (\a_P^G)^*}\intop_\Omega \Eis(\rho(P,\lambda,h_{R,q})\varphi,\lambda)(x)\overline{\Eis(\varphi,\lambda)(x)}  \d x \d \lambda.
    \end{align*}
    Note that absolute convergence of the spectral side \cite{finis2011spectral-abs} allows us to pull the $\Omega$-integral inside. Here we choose a basis $\calB_\pi(P)$ such that each of its elements is either $K_\infty^0K(q)$-invariant, or orthogonal to the $K_\infty^0K(q)$-invariant subspace.
    
    From the definition of $h_{R,q}$ we note that the operator $\rho(P,\lambda,h_{R,q})$ projects $\calA^2_\pi(P)$ onto $\calA^2_\pi(P)^{K_\infty^0K(q)}$, and acts on it by the scalar $(\tilde{h}_R\cdot \tilde{h}^*_R)(\mu_\pi+\lambda)=|\tilde{h}_R(\mu_\pi+\lambda)|^2$.
    Focusing on the inner-most integral above and assuming that $\varphi$ is $K_\infty^0K(q)$-invariant, we have
    \begin{equation*}
     \intop_\Omega \Eis(\rho(P,\lambda,h_{R,q})\varphi,\lambda)(x)\overline{\Eis(\varphi,\lambda))(x)}  \d x 
     =  |\tilde{h}_R(\mu_\pi+\lambda)|^2\intop_\Omega|\Eis(\varphi,\lambda)(x)|^2 \d x 
    \end{equation*}
    On the other hand, by property (2) of \cref{sec:truncation}, we have
    \begin{equation*}
    \intop_\Omega |\Eis(\varphi,\lambda)(x)|^2 \d x 
    \le  \intop_{G(\Q) \backslash G(\bbA)} |\Lambda^T\Eis(\varphi,\lambda)(x)|^2 \d x
    \end{equation*}
     for $d(T)$ sufficiently large.
    Plugging this into the above integral and reverse engineering the above manipulation with $h_{R,q}$ we deduce that 
    \begin{align*}
    &\intop_{\Omega} K_{f}(x,x) \d x \\
    \le& \sum_{P\supset P_0} \frac{1}{n_P} \sum_{\pi \in \Pi_2(M_P)} \sum_{\varphi  \in \calB_\pi(P)} \intop_{i (\a_P^G)^*} \langle \Lambda^T\Eis(\rho(P,\lambda,h_{R,q})\varphi,\lambda), \Lambda^T\Eis(\varphi,\lambda) \rangle_{\calA^2(G)} \d \lambda.
    \end{align*}
    Now the statement of the proposition follows from \cref{prop:spectral-side-full}.
\end{proof}

\begin{cor}\label{cor: AB reduction2}
Let $q$ be square-free. There is an $x\in \Omega$ such that
\[
K_{h_{R,q}}(x,x) \ll_\epsilon (Rq)^\epsilon q(1+R^{-n(n-1)}q^{n^2-1}),
\]
for every $\epsilon>0$
\end{cor}
\begin{proof}
    The corollary is immediate from \cref{prop: AB reduction1} and \cref{prop:assing-blomer}.
\end{proof}

We move the discussion to the classical language. We let $K_{h_{R}*h_{R}^*}^q(x,y)$ be the classical automorphic kernel of $h_R*h_R^*$, i.e., 
\[
K_{h_R*h_R^*}^q(x,y) = \sum_{\gamma \in \Gamma(q)} (h_R*h_R^*)(x^{-1}\gamma y),
\]
where $x,y \in \Gamma(q) \backslash \SL_n(\R)$.

\begin{lemma}\label{lem: AB reduction3}
    Let $q$ be square-free, there is an $x\in \Gamma(q) \backslash B_{R_0} \subset \Gamma(q) \backslash \SL_n(\R)$ such that
    \[
K_{h_R*h_R^*}^q(x,x) \ll_\epsilon (Rq)^\epsilon(q^{-(n^2-1)}+R^{-n(n-1)}),\]
for every $\epsilon>0$.
\end{lemma}

\begin{proof}
This follows from the classical-adelic dictionary. 
Let $x_0 \in B_{R_0} \times K(1) \subset G(\bbA)^1$ be a representative point from \cref{cor: AB reduction2}, and let $x := x_{0,\infty} \in B_{R_0} \subset \SL_n(\R)$ be its $\infty$-coordinate.

We claim that
\[
K_{h_{R,q}}(x_0,x_0) = \vol(K(q))^{-1} K_{h_R*h_R^*}^q(x,x),
\]
which implies the claim together with \cref{cor: AB reduction2} and the fact that $\vol(K(q)) \ll_\epsilon q^{-n^2+\epsilon}$.

Indeed,
\begin{equation*}
K_{h_{R,q}}(x_0,x_0) = \sum_{\gamma \in G(\Q)} h_{R,q}(x_0^{-1}\gamma x_0)   = \sum_{\gamma \in G(\Q)}(h_R*h_R)^*(x_{0,\infty}^{-1}\gamma x_{0,\infty})h_q(x_{0,f}^{-1}\gamma x_{0,f})
\end{equation*}

An element $\gamma \in G(\Q)$ that has non-zero contribution to the above sum must satisfy $x_{0,\infty}^{-1}\gamma x_{0,\infty} \in \SL_n(\R)$, so $\gamma \in \SL_n(\R)$. Also, it must hold that $x_{0,f}^{-1}\gamma x_{0,f} \in K(q)$. In addition, since $x_{0,f}\in K(1)$ and $K(q)$ is normal in $K(1)$, it holds that $\gamma \in K(q)$. Therefore $\gamma \in \Gamma(q) \subset \SL_n(\Z)$, and $h_q(x_{0,f}^{-1}\gamma x_{0,f})=\frac{1}{\vol(K(q))}$. Therefore,
\[
K_{h_{R,q}}(x_0,x_0) = \vol(K(q))^{-1}\sum_{\gamma \in \Gamma(q)}(h_R*h_R)^*(x^{-1}\gamma x)
= \vol(K(q))^{-1} K_{h_R*h_R^*}^q(x,x),
\]
and the proof is complete.
\end{proof}

We now claim that we may replace $x \in B_{R_0}$ with $x=e$ in \cref{lem: AB reduction3}.

\begin{prop}\label{prop:AB main prop}
Let $q$ be square-free. It holds that
\[
K_{h_{R}*h_{R}^*}^q(e,e) \ll_\epsilon (Rq)^\epsilon(q^{-(n^2-1)}+R^{-n(n-1)}),\]
for every $\epsilon>0$.
\end{prop}

\begin{proof}
Let $x$ be the point from \cref{lem: AB reduction3}. We choose for $x$ a representative in $B_{R_0}$. This implies that for $C>0$ sufficiently large and for every $g\in \SL_n(\R)$ we have
\[
\|x^{-1}g\| \le C\|g\|, \quad \|gx\| \le C\|g\|.
\]
From the definition of $h_{R}$, the above implies that for $c>0$ sufficiently small (e.g., $c=C^{-1}/2$) and every $g\in \SL_n(\R)$
\[
h_{cR}(g) \ll h_R(x^{-1}g), \quad
h_{cR}^*(g) \ll h_R^*(gx),
\]
and therefore for every $\gamma \in \Gamma(q)$,
\[
(h_{cR}*h_{cR}^*)(\gamma)\ll (h_R*h_R^*)(x^{-1}\gamma x).
\]
We deduce using \cref{lem: AB reduction3} that 
\[
K_{h_{cR}*h_{cR}^*}^q(e,e) \ll K_{h_{R}*h_{R}^*}^q(x,x) \ll_\epsilon (Rq)^\epsilon(q^{-(n^2-1)}+R^{-n(n-1)}).\]
After changing $R$ to $cR$ the proposition follows.
\end{proof}

\subsection{Proof of \texorpdfstring{\cref{thm:counting}}{Theorem 4}}

We adopt the notations developed for the classical spectral decomposition of $L^2(\Gamma(q) \backslash \SL_n(\R) / K_\infty^0)$ in \cite[\S2.5]{assing2022density}. We write
the spectral decomposition of the kernel of $h \in  C_c^\infty(K_\infty^0 \backslash \SL_n(\R) / K_\infty^0)$ as
\begin{equation}\label{spectral-decomposition-classical}
    K_h^q(x,y) = \intop_{\Gamma(q)} \tilde{h}(\mu_\varpi) \varpi(x)\overline{\varpi(y)}\d \varpi,
\end{equation}
    where $\intop_{\Gamma(q)} V_\varpi \d \varpi$ is a spectral decomposition of $L^2(\Gamma(q) \backslash \SL_n(\R) / K_\infty^0)$, respecting the unramified Hecke algebra (see \cite[\S 2.5]{assing2022density} for more details). This classical spectral decomposition may be deduced from the more precise adelic spectral decomposition (e.g., \cite[Equation (7.6)]{arthur2005intro}). The value $\mu_\varpi$ is the Langlands parameter of the $K_\infty^0$-invariant representation $V_\varpi$.
    
\begin{lemma}\label{lem:counting pf1}
We have
\[K^q_{h*h}(e,e)\le K^q_{h*h^*}(e,e)\]
for any $h\in C_c^\infty(K_\infty^0 \backslash \SL_n(\R) / K_\infty^0)$.
\end{lemma}

\begin{proof}
We have
    \begin{equation*}
    K_{h*h}^q(e,e) = 
    \intop_{\Gamma(q)}\tilde{h}(\mu_\varpi)^2|\varpi(e)|^2\d\varpi  \le \intop_{\Gamma(q)}|\tilde{h}(\mu_\varpi)|^2|\varpi(e)|^2\d\varpi
     = K_{h*h^*}^q(e,e),
    \end{equation*}
where the first and last equalities follow from the spectral decompositions of the corresponding kernels as in \cref{spectral-decomposition-classical}. 
\end{proof}

\begin{proof}[Proof of \cref{thm:counting}]
Let $h_{R}^{(1)} := \frac{1}{\vol(B_R)}\One_{B_R}$ be the normalized characteristic function of the ball $B_R$.  
It holds for every $g\in \SL_n(\R)$ that
    $h_{R}^{(1)}(g)\ll h_{R}(g)$, which, upon applying \cref{lem:counting pf1} and \cref{prop:AB main prop}, implies that
\begin{equation}\label{eq:bound-of-char-kernel}
    K_{h_{R}^{(1)}*h_{R}^{(1)}}^q(e,e) \ll K_{h_{R}*h_{R}}^q(e,e) \le K_{h_{R}*h_{R}^*}^q(e,e) \ll_\epsilon (Rq)^\epsilon(q^{-(n^2-1)}+R^{-n(n-1)}).
\end{equation}
Notice that the kernel $K_{h_{R}^{(1)}*h_{R}^{(1)}}^q(e,e)$ is well defined even tough $h_{R}^{(1)}$ is not smooth. The same will be true for other functions below.

    Next, we claim that for $c>0$ sufficiently small and $R$ sufficiently large, it holds that for every $g\in \SL_n(\R)$
    \begin{equation}\label{eq: char function vs conv}
    h_{cR^2}^{(1)}(g) \ll (h_{R}^{(1)}*h_{R}^{(1)})(g).    
    \end{equation}
    To prove this claim, notice that \[\One_{B_R}*\One_{B_R}(g)=\vol(B_R \cap gB_R^{-1}).\]
    For $c>0$ small enough and $R$ large enough the above is $\gg 1$ as long as $\|g\| \le c R^2$.
    Indeed, we may assume that $g$ is diagonal with positive values on the diagonal. Then $B_R \cap gB_R^{-1}$ contains all the element of the form $g^{1/2} x$, for $\|x\| \le C$ (for $c>0$ sufficiently small relative to $C$).
    Therefore, for every $g\in \SL_n(\R)$,
    \[
    \One_{B_{cR^2}}(g) \ll (\One_{R}*\One_{R})(g).
    \]
    Using \cref{eq:vol of B_R} we deduce that \cref{eq: char function vs conv} holds. Thus using \cref{eq:bound-of-char-kernel} we obtain
    \begin{align*}
    \frac{1}{\vol(B_{R^2})}\sum_{\gamma \in \Gamma(q)}\One_{B_{R^2}}(\gamma) = K_{h_{R^2}^{(1)}}^q(e,e) 
    \ll (Rq)^\epsilon(q^{-(n^2-1)}+R^{-n(n-1)}).
    \end{align*}
Using \cref{eq:vol of B_R} once again and changing $R$ to $\sqrt{R}$ we complete the proof.  
\end{proof}

\subsection{Proof \texorpdfstring{\cref{thm:optimal-lifting}}{Theorem 5}}

To prove the theorem we first prove a couple of lemmata. Fix $\eta>0$. We consider the function $m_R:=h_R*h_{R^\eta}$.

\begin{lemma}\label{lem:lifting lemma 1}
There is a $\delta>0$ depending only on $\eta$ such that for every $q$ square-free and $R\ge q^{1+1/n}$
    \[
    K_{m_R*m_R^*}^q(e,e) - \frac{1}{\vol(\Gamma(q) \backslash \SL_n(\R))} \ll  R^{-\delta}q^{-(n^2-1)}.
    \]
\end{lemma}

\begin{proof}
We apply the spectral decomposition as in \cref{spectral-decomposition-classical}
to obtain 
\[
K_{m_R*m_R^*}^q(e,e)= \intop_{\Gamma(q)} |\tilde{m}_R(\mu_\varpi)|^2 |\varpi(e)|^2 \d \varpi. 
\]
Among the automorphic forms $\varpi$ appearing in the spectral decomposition we have the $L^2$-normalized constant function $\varpi_\triv$, for which it holds that
\[
|\varpi_\triv(e)|^2 = \frac{1}{\vol(\Gamma(q) \backslash \SL_n(\R))}.
\]
Also, using the facts that $\tilde{m}_R(\lambda) = \tilde{h}_R(\lambda) \tilde{h}_{R^\eta}(\lambda)$ and $\tilde{h}_R(\mu_{\varpi_\triv})=\|h_R\|_1=1$ it follows that $\tilde{m}_R(\mu_{\varpi_\triv}) =1=\|m_R\|_1$.

On the other hand, there is a $\delta_1>0$ such that for every $\varpi \ne \varpi_\triv$, it holds that $\|\Re(\mu_\varpi)\|_\infty \le (n-1)/2-\delta_1$. This follows from the explicit property (T) for $n\ge 3$ (where we may actually take $\delta_1 =1/2$), and from Selberg's spectral gap theorem for $n=2$; see \cite{sarnak2005notes}. 

Using \cref{non-temp-paley-wiener}, we deduce that there is a $\delta_2>0$ depending only on $\eta$ (and $n$) such that for $\varpi\ne \varpi_\triv$
\[
|\tilde{m}_R(\mu_{\varpi})|\le R^{-\delta_2} |\tilde{h}_R(\mu_{\varpi})|.
\]
Combining all the information above we obtain
\begin{align*}
&K_{m_R*m_R^*}^q(e,e) - \frac{1}{\vol(\Gamma(q) \backslash \SL_n(\R))}
 = \intop_{\varpi \ne \varpi_\triv} |\tilde{m}_R(\mu_\varpi)|^2 |\varpi(e)|^2 \d \varpi \\
& \le R^{-2\delta_2} \intop_{\varpi \ne \varpi_\triv} |\tilde{
h}_R(\mu_\varpi)|^2 |\varpi(e)|^2 \d \varpi  \le R^{-2\delta_2} K_{h_R*h_R^*}^q(e,e).
\end{align*}
The claim now follows from \cref{prop:AB main prop}.
\end{proof}

\begin{lemma}\label{lem:lifting lemma 2}
Let $M_R^q\colon \Gamma(q) \backslash \SL_n(\R) \to \R$ be the function $M_R^q(x):= \sum_{\gamma \in \Gamma(q)} m_R(\gamma x)$.
Then 
\[
K_{m_R*m_R^*}^q(e,e) = \| M_R^q\|_2^2.
\]
\end{lemma}

\begin{proof}
We see that
\begin{equation*}
K_{m_R*m_R^*}^q(e,e)
= \sum_{\gamma \in \Gamma(q)} (m_R*m_R^*)(\gamma) 
=\sum_{\gamma\in \Gamma(q)} \intop_{x\in \SL_n(\R)}m_R(x)m_R(\gamma^{-1} x) \d x
\end{equation*}
Folding the integral over $\Gamma(q)$ we obtain the above equals
\[
\intop_{x\in \Gamma(q) \backslash \SL_n(\R)}\sum_{\gamma_1\in \Gamma(q)} m_R(\gamma_1 x)\sum_{\gamma\in \Gamma(q)}m_R(\gamma^{-1} \gamma _1 x) \d x = \intop_{x\in \Gamma(q) \backslash \SL_n(\R)} M_R^q( x)M_R^q( x) \d x.
\]
We conclude that the above is $\|M_R^q\|^2_2$ as $m_R$ is real-valued.
\end{proof}

\begin{proof}[Proof of \cref{thm:optimal-lifting}]
Let $U_q \in L^2 (\Gamma(q) \backslash \SL_n(\R))$ be the $L^1$-normalized constant function on $\Gamma(q) \backslash \SL_n(\R)$. Then by \cref{lem:lifting lemma 2}, the fact that 
\[
\intop_{\Gamma(q) \backslash \SL_n(\R)} M_R^q(x) \d x = \intop_{\SL_n(\R)} m_R(x) \d x=1,
\]
and \cref{lem:lifting lemma 1} we have that there is a $\delta>0$ depending only on $\eta>0$ such that for $R\ge q^{1+1/n}$
\begin{equation}\label{eq:pf-optimal-3}
\|M_R^q - U_q \|_2^2 = \|M_R^q \|_2^2 - \frac{1}{\vol(\Gamma(q) \backslash \SL_n(\R))} \ll R^{-\delta}q^{-(n^2-1)}.
\end{equation}

We identify $\overline{g}\in \Gamma(q)\backslash \Gamma(1) \cong \SL_n(\Z/q\Z)$ with some representative of it in $\Gamma(1):=\SL_n(\Z)\subset \SL_n(\R)$.
We notice that if $\overline{g}$ has no lift to $\Gamma(1)$ of size bounded by $C_1R^{1+\eta}$ for certain sufficiently large but fixed $C_1>0$ then for all $\gamma\in\Gamma(q)$ we have
\[\|\gamma \overline{g} w\|\ge \|\gamma \overline{g}\| \|w^{-1}\|^{-1}\ge C_3R^{1+\eta},\]
for certain $C_3>0$ sufficiently large,
where $w\in\Gamma(1)\backslash\SL_n(\R)$ with $\|w\|\le C_2$ for certain $C_2\ge n$.
Consequently, noting that the support of $m_R$ is on $g\in \SL_n(\R)$ such that $\|g\| \le C_4R^{1+\eta}$ for certain fixed $C_4>0$ we conclude that
\[M_R^q(\overline{g}w)=0.\]
Thus for such $\overline{g}$ we have
\[
\intop_{\substack{w\in\mathcal{F}\\ \|w\|\le C_2}} \left|M_R^q(\overline{g}w)-\frac{1}{\vol(\Gamma(q)\backslash \SL_n(\R))}\right|^2 \d w \gg \frac{1}{\vol(\Gamma(q)\backslash \SL_n(\R))^2} \gg \frac{1}{q^{2(n^2-1)}},
\]
where $\mathcal{F}\subseteq\SL_n(\R)$ is the standard (\emph{i.e.} containing identity) fundamental domain of $\Gamma(1)\backslash \SL_n(\R)$.
Summing the above over such $\overline{g}$ we obtain
\[
\frac{N(q,C_1R^{1+\eta})}{q^{2(n^2-1)}} \ll \sum_{\overline{g} \in \Gamma(q) \backslash \Gamma(1)} \intop_{\mathcal{F}} \left|M_R^q(\overline{g}w)-\frac{1}{\vol(\Gamma(q)\backslash \SL_n(\R))}\right|^2 \d w,
\]
where $N(q,R)$ is the number of $\overline{g} \in \SL_n(\Z/q\Z)$ without a lift to $\Gamma(1)$ of norm bounded by $R$. Note that
\[\sum_{\overline{g} \in \Gamma(q) \backslash \Gamma(1)}\left|M_R^q(\overline{g}w)-\frac{1}{\vol(\Gamma(q)\backslash \SL_n(\R))}\right|^2\]
as a function of $w$ is left-invariant by $\Gamma(1)$. Thus we can write the $w$-integral of the above over $\mathcal{F}$ as
\begin{equation*}
\sum_{\overline{g} \in \Gamma(q) \backslash \Gamma(1)} \intop_{\Gamma(1)\backslash \SL_n(\R)} \left|M_R^q(\overline{g}w)-\frac{1}{\vol(\Gamma(q)\backslash \SL_n(\R))}\right|^2 \d w =\|M_R^q - U_q \|_2^2.    
\end{equation*}
On the other hand, for $R \ge q^{1+1/n}$ we have
\begin{equation*}
\|M_R^q - U_q \|_2^2 \ll R^{-\delta}q^{-(n^2-1)},    
\end{equation*}
which follows from \cref{eq:pf-optimal-3}.
Finally, plugging in $R= q^{1+1/n}$ we get that for $C>0$ sufficiently large
\[
N(q,C_1q^{1+1/n+\eta(1+1/n)}) \ll q^{n^2-1-\delta}.
\]
Modifying $\eta$ we conclude the proof.
\end{proof}

\subsection{Proof of \texorpdfstring{\cref{prop:assing-blomer}}{Assing-Blomer Proposition}}

We start by applying the second formulation of $J^T(h)$ as in \cref{prop:spectral-side-full}. Note that $\rho(P,\lambda, h_{R,q})\mid_{\calA^2_\pi(P)}$ projects onto $\calA^2_\pi(P)^{K_\infty^0 K(q)}$. Thus working as in the proof of \cref{thm:BCF} it suffices to show that
\begin{multline}\label{reduced-assing-blomer}
    \sum_{P\supset P_0}\max_{L\in\calL(M_P)}\sum_{\pi\in\Pi_2(M_P)}\dim\left(\calA_{\pi}^2(P)^{K_\infty^0 K(q)}\right)\\
    \times\intop_{i(\a_{L}^G)^*}|\tilde{h}_{R}(\mu_\pi+\lambda)|^2\left\|\calM_{L}^T(P,\lambda)|_{\calA_{\pi}^2(P)^{K_\infty^0 K(q)}}\right\|_{\op}\d \lambda \ll_\epsilon \|T\|^{n-1}(Rq)^\epsilon(q+R^{-n(n-1)}q^{n^2}).
\end{multline}
We apply \cref{non-temp-paley-wiener} to bound the inner integral on the left-hand side of the above by
\[\ll_{N,\epsilon} R^{-n(n-1)+2n\|\Re(\mu_\pi)\|_\infty+\epsilon}\intop_{i(\a_{L}^G)^*}(1+\|\mu_\pi+\lambda\|)^{-N}\left\|\calM_{L}^T(P,\lambda)|_{\calA_{\pi}^2(P)^{K_\infty^0 K(q)}}\right\|_{\op}\d \lambda.\]
As $\pi$ is a representation of $M(\bbA)^1$ and $\lambda\in i(\a_M^G)^*$ we have $\langle \mu_\pi,\lambda\rangle =0$, consequently, we have $\|\mu_\pi+\lambda\|^2\asymp \|\mu_\pi\|^2+\|\lambda\|^2$. Thus the above integral can be bounded by
\[(1+\|\mu_\pi\|)^{-N/2}\intop_{i(\a_{L}^G)^*}(1+\|\lambda\|)^{-N/2}\left\|\calM_{L}^T(P,\lambda)|_{\calA_{\pi}^2(P)^{K_\infty^0 K(q)}}\right\|_{\op}\d \lambda.\]
To estimate the above integral we proceed exactly as in the proof \cref{thm:BCF}. We first write it as the sum of the sub-integrals over $m-1\le\|\lambda\|\le m$ where $m$ varies over $\N$. We further majorize each sub-integral by a sum of integrals over $\|\lambda-\eta\|\le 1$ with $\eta\in i(\a^G_L)^*$ and $\|\eta\|\ll m$. Finally, we apply \cref{prop:calM bound} to bound each integral after making $N$ large enough. This gives us \[\intop_{i(\a_{L}^G)^*}(1+\|\lambda\|)^{-N/2}\left\|\calM_{L}^T(P,\lambda)|_{\calA_{\pi}^2(P)^{K_\infty^0 K(q)}}\right\|_{\op}\d \lambda\ll \left(\|T\|\log(1+\|\mu_\pi\|+\level(K(q))\right)^{\dim\a^G_L}.\]
Hence we obtain
\begin{multline}\label{lambda-integral-assing-blomer}
    \max_{L\in\calL(M_P)}\intop_{i(\a_{L}^G)^*}|\tilde{h}_{R}(\mu_\pi+\lambda)|^2\left\|\calM_{L}^T(P,\lambda)|_{\calA_{\pi}^2(P)^{K_\infty^0 K(q)}}\right\|_{\op}\d \lambda\\
    \ll_{\epsilon,N} \|T\|^{n-1}(Rq)^\epsilon R^{-n(n-1)+2n\|\Re(\mu_\pi)\|_\infty}(1+\|\mu_\pi\|)^{-N},
\end{multline}
for all sufficiently large $N>0$. 

Now we move on to estimate the $\pi$-sum in \cref{reduced-assing-blomer}. First, let us estimate the contribution for $\|\mu_\pi\|\ge (Rq)^\epsilon$. Dividing into dyadic intervals, and using \cref{lambda-integral-assing-blomer} and the fact that $\|\Re(\mu_\pi)\|_\infty\ll 1$ we estimate this contribution by
\[\ll(Rq)^{O(1)}\sum_{m\ge (Rq)^\epsilon}\sum_{\substack{\pi\in\Pi_2(M_P)\\ \|\mu_\pi\|\asymp m}}\dim\left(\calA_{\pi}^2(P)^{K_\infty^0 K(q)}\right)\|\mu_\pi\|^{-N}.\]
There is an absolute $K$ such that the inner sum is bounded by $\ll m^{-N}(mq)^K$, which follows from \cite[eq.(7.2)]{assing2022density}, classification of discrete series \cite{moeglin1989residual}, and local Weyl law \cite{muller2007weyl} (also see the discussion in \cite[pp.35]{assing2022density}). Making $N$ large enough we can bound the above double sum by $\ll_N(Rq)^{-N}$. Thus we may restrict the $\pi$-sum in \cref{reduced-assing-blomer} with $\|\mu_\pi\|\le (Rq)^\epsilon$. We conclude the proof by applying \cref{lem:assing-blomer} below.

\begin{lemma}\label{lem:assing-blomer}
We have
\[\sum_{P\supset P_0}\sum_{\substack{\pi\in\Pi_2(M_P)\\ \|\mu_\pi\|\le (Rq)^\epsilon}}\dim\left(\calA_{\pi}^2(P)^{K_\infty^0 K(q)}\right)R^{2n\|\Re(\mu_\pi)\|_\infty}\ll_\epsilon (Rq)^\epsilon q(R^{n(n-1)}+q^{n^2-1}),\]
with $R,q$ as in \cref{prop:assing-blomer}.
\end{lemma}

The statement and the proof below are quite similar to those of \cite[Theorem 7.1]{assing2022density}. To make the similarities comprehensible we make the dictionary that $T$ and $M$ in \cite{assing2022density} are $R^n$ and $(Rq)^\epsilon$, respectively, in our case.

\begin{proof}
First, trivially we have
\begin{multline*}
    \sum_{P\supset P_0}\sum_{\substack{\pi\in\Pi_2(M_P)\\ \|\mu_\pi\|\le (Rq)^\epsilon}}\dim\left(\calA_{\pi}^2(P)^{K_\infty^0 K(q)}\right)R^{2n\|\Re(\mu_\pi)\|_\infty}\\
    \le \sum_{P\supset P_0}\sum_{\substack{\pi\in\Pi_2(M_P)\\ \|\mu_\pi\|\le (Rq)^\epsilon}}\dim\left(\calA_{\pi}^2(P)^{K_\infty^0 K(q)}\right)R^{2n\|\Re(\mu_\pi)\|_\infty}\intop_{\substack{\lambda\in i(\a_P^G)^*\\ \|\lambda\|\le (Rq)^\epsilon}}\d\lambda.
\end{multline*}
In the classical language as in \cite[Theorem 7.1]{assing2022density} the right-hand side above is bounded by 
\begin{equation}\label{eq:classical-AB}
q\int_{\substack{\Gamma(q)\\ \|\mu_\varpi\|\le 2(Rq)^\epsilon}} R^{2n\|\Re(\mu_\varpi)\|_\infty}\d\varpi.
\end{equation}
For the relevant notations, we refer to \cite[\S2]{assing2022density}. The factor $q$ comes from the $\varphi(q) \le q$ possible twists by Dirichlet characters when moving from the adelic to the classical language (see \cite[\S2.6.2]{assing2022density} and also \cite[\S3.4]{lapid2009spectral}).

If $R\le R_0:=q^{1+1/n}(Rq)^{-\epsilon}$ for sufficiently small $\epsilon>0$ we directly apply \cite[Theorem 7.1]{assing2022density} to bound the integral in \cref{eq:classical-AB} by $\ll_\epsilon (Rq)^\epsilon q^{n^2-1}$.

If $R> R_0$ then we work as in \cite[\S 8.2]{assing2022density}. Note that $\|\Re(\mu_\varpi)\|_\infty\le \frac{n-1}{2}$ (which is attained when $\varpi$ is one dimensional). In this case, we write the integral in \cref{eq:classical-AB} as
\[(R/R_0)^{n(n-1)}\int_{\substack{\Gamma(q)\\ \|\mu_\varpi\|\le 2(Rq)^\epsilon}} R_0^{2n\|\Re(\mu_\varpi)\|_\infty}\d\varpi\ll_\epsilon (Rq)^\epsilon R^{n(n-1)}.\]
This completes the proof.
\end{proof}

\bibliographystyle{acm}
\bibliography{./database}

\end{document}